\documentclass[12pt]{amsart}
\usepackage{latexsym}
\usepackage{amsmath,amsthm,amsfonts,amscd,eucal,amssymb}
\usepackage{epsfig}

\newtheorem{thm}{Theorem}[section]

\newtheorem{lem}[thm]{Lemma}
\newtheorem{claim}[thm]{Claim}
\newtheorem{prop}[thm]{Proposition}
\theoremstyle{definition}
\newtheorem{dfn}[thm]{Definition}
\theoremstyle{remark}
\newtheorem{rem}[thm]{Remark}
\numberwithin{equation}{section}

\def\ca{{\mathcal A}}

\def\ce{{\mathcal E}}
\def\cf{{\mathcal F}}

\def\ch{{\mathcal H}}

\def\ct{{\mathcal T}}

\def\cv{{\mathcal V}}

\newcommand{\bn}{\mathbb N}
\def\br{{\mathbb R}}
\def\bz{{\mathbb Z}}

\renewcommand{\a}{\alpha}
\def\b{\beta}
\def\g{\gamma}        \def\G{\Gamma}
\def\d{\delta}        \def\D{\Delta}
\def\eps{\varepsilon}

    \def\Th{\Theta}

\def\l{\lambda}

\def\r{\rho}
\def\s{\sigma}       \renewcommand{\S}{\Sigma}
    
\def\t{\tau}

\def\f{\varphi}

\def\o{\omega}        \def\O{\Omega}

\newcommand{\set}[1]{\left\{#1\right\}}

\def\ov{\overline}
\def\wt{\widetilde}
\DeclareMathOperator{\osc}{Osc}

\DeclareMathOperator{\ab}{Ab}
\DeclareMathOperator{\deck}{deck}
\DeclareMathOperator{\V}{Vec}

\def\dg{z}
\def\uforme{\O^1(\cf)}                                         
\def\forme{\O^1(K)}                                          
\def\ovforme{\ov{\forme}^{\mathrm{sup}}} 
\def\qex{\O^1_{\rm{loc}}(K)}                           
\def\toplex{\O^1_{\rm{loc}}C(K)}                    
\def\KK{\widetilde{K}}
\def\LL{\widetilde{L}}
\def\ds{\displaystyle}
\def\UUAC{Uniform Universal Abelian Covering }
\def\bordo{\partial}
\def\perim{\pi}
\def\normaQ{Q^{1/2}}
\def\normaInf#1{ \| #1 \|_{\mathrm{sup}} }

\begin{document}

\title[1-forms on Gasket]
{Integrals and Potentials of differential 1-forms on the Sierpinski Gasket}
\author{Fabio Cipriani}
\address{(F.C.) Politecnico di Milano, Dipartimento di Matematica,
piazza Leonardo da Vinci 32, 20133 Milano, Italy.}%
\email{fabio.cipriani@polimi.it}%
\author{Daniele Guido}
\address{(D.G.) Dipartimento di Matematica, Universit\`a di Roma ``Tor
Vergata'', I--00133 Roma, Italy}
\email{guido@mat.uniroma2.it}
\author{Tommaso Isola}
\address{(T.I.) Dipartimento di Matematica, Universit\`a di Roma ``Tor
Vergata'', I--00133 Roma, Italy}
\email{isola@mat.uniroma2.it}
\author{Jean-Luc sauvageot}
\address{(J.-L.S.) Institut de Math\'ematiques, CNRS-Universit\'e Pierre et Marie
Curie, boite 191, 4 place Jussieu, F-75252 Paris Cedex 05, France}
\email{jlsauva@math.jussieu.fr}

\thanks{This work has been partially supported by GNAMPA, MIUR,
the European Networks ``Quantum Spaces - Noncommutative Geometry"
HPRN-CT-2002-00280, and ``Quantum Probability and Applications to
Physics, Information and Biology'', GDRE GREFI GENCO, and  the ERC Advanced Grant 227458 OACFT ``Operator Algebras and Conformal Field Theory"}

\begin{abstract}
We provide a definition of integral, along paths in the Sierpinski gasket $K$, for differential smooth 1-forms associated to the standard Dirichlet form $\ce$ on $K$.
We show how this tool can be used to study the potential theory on $K$. In particular, we prove: i) a de Rham reconstruction of a 1-form from its periods around lacunas in $K$; ii) a Hodge decomposition of 1-forms with respect to the Hilbertian energy norm; iii) the existence of potentials of smooth 1-forms on a suitable covering space of $K$. We
finally show that this framework provides versions of the de Rham duality theorem for the fractal $K$.
\end{abstract}
\maketitle

\section{Introduction}
\subsection{Purpose of the work}
The aim of this work is to develop, on the fractal set $K$ known as {\it Sierpinski gasket}, a notion of {\it line integral}
\[
\int_\gamma \omega
\]
along oriented paths $\gamma$ in $K$ for a class of {\it differential 1-forms} $\omega$ on $K$. The purpose for doing this is twofold: on the one hand, we wish to set up tools useful to contribute to the potential theory of $K$, studied in particular by Kigami \cite{Kiga}, Strichartz \cite{St06}, see also the recent work by Koskela and Zhou \cite{KZ}; on the other hand, our intention is to use them to construct local representations, i.e. by integrals, of topological invariants of $K$.
\par\noindent
Our approach is based on the existence of a differential calculus underlying any regular Dirichlet space $X$, developed in \cite{CiSa}, \cite{C} and further investigated in \cite{CiSa2} and \cite{IRT} for fractal spaces.
There, the  differential bimodule of universal 1-forms $\uforme$ on  the algebra of finite energy functions $\cf$ on $X$ is endowed with a quadratic form $Q$ associated with the Dirichlet energy. By separation and completion one gets a Hilbert $\cf$-bimodule $\ch$, called the {\it tangent module} of $\ce$,  whose elements are termed {\it differential 1-forms} on $X$, together with a derivation $\partial :\cf\to\ch$, i.e. a map satisfying the Leibniz rule
$\partial (fg)=(\partial f)g + f(\partial g)$ $ f,g\in\cf$.
Such derivation is a differential square root of the Dirichlet form in the sense that
\[
\ce[a]=\|\partial a\|^2_\ch\qquad a\in\cf\, .
\]
Our main results are: define the integral of elements of $\uforme$ along (a suitable class of) oriented paths in $K$, and show that this integral passes to the quotient w.r.t. $Q$, hence is well defined on the space $\forme:=\uforme/\{Q=0\}$, whose elements we call smooth 1-forms;
establish de Rham first and second Theorems, by proving that the sequence of periods around lacunas gives rise to a unique {\it harmonic form}; prove a Hodge Theorem, namely each cohomology class in a suitably defined de Rham cohomology grop $H^1 (K)$ has a unique harmonic representative; the establishment of a pairing between the cohomology of forms and the  \v{C}ech homology group of the gasket (de Rham duality theorem); and finally, the construction of an (abelian) projective covering, related to the Uniform Universal Covering introduced in \cite{BerPla}, where  potentials of 1-forms will be defined.

\medskip

The classical framework we refer to is that of harmonic integrals on differentiable manifolds, developed by de Rham \cite{dR} and  Hodge \cite{Hod}. There, the notions of differential 1-form and line integral are direct outcome of the notion of tangent bundle. The analytic tool of exterior differentiation of forms then naturally provides homotopy invariants by means of the differential complex and its associated cohomology groups.
The notion of line integral on the manifold $M$ allows to establish a local pairing first between closed 1-forms and 1-cycles, and then between the first de Rham cohomology group $H^1 (M)$ and the first singular homology group $H_1 (M)$. Furthermore, the choice of a Riemannian metric on M allows to introduce the notions of co-closed and harmonic forms in such a way that each cohomology class in $H^1 (M)$ has a unique harmonic representative.

\medskip

Trying to develop the above framework on the Sierpinski gasket $K$, two main problems have to be tackled.

The first is that $K$ is not a manifold: it was originally introduced in \cite{S} as an example of space with a dense set of ramification points so that it has no open sets homeomorphic to Euclidean domains.
This is the reason why a notion of differentiable structure on $K$ has to be introduced in an unconventional way. We choose to do so by using the notion of energy or Dirichlet form, a sort of generalized Dirichlet integral, developed by  Beurling and  Deny \cite{BD}, that can be considered on locally compact Hausdorff spaces. In particular, we consider the so called standard Dirichlet form $\ce$ considered by Kusuoka \cite{Ku} in his construction of a diffusion process on $K$ and studied by Fukushima and Shima \cite{FS} and by Kigami \cite{Kiga} in his framework of harmonic theory on self-similar fractal sets like $K$.
The primary role of $\ce$ is to provide the class of finite energy functions $\cf$, which is a dense subalgebra of the algebra of continuous functions $C(K)$, and plays the role of a Sobolev space on the gasket. More importantly, there exists a canonical first order differential calculus associated to Dirichlet forms, as developed in \cite{CiSa}. It is represented by a closed derivation $\partial :\cf\to\ch$, defined on $\cf$ with values in a Hilbert $C(K)$-module $\ch$, which is a differential square root of the Dirichlet form in the sense that $\ce [a]=\|\partial a\|^2_\ch$.
We notice that this differential structure, namely the module $\ch$ and the derivation $\partial$, is essentially unique and only depends upon the quadratic form $\ce$ defined on $\cf$ and not on the choice of a reference measure on $K$.
One of the main technical issues will be the proof that the integral along oriented paths makes sense on (suitably regular) elements of $\ch$.
As we shall see below, this will force us to a long detour: the introduction of  the bimodule of universal 1-forms $\uforme$ on the Dirichlet algebra $\cf$, the definition of line-integrals on it, and then the proof that an element of $\uforme$ with zero Hilbert norm has zero integral along all edges, namely the integral makes sense on the quotient. What we get then is an $\cf$-module $\forme$, which densely embeds in $\ch$, thus furnishing the smooth subspace on which line integrals make sense.

The second problem is that $K$ is a topological space which is not semilocally simply connected, so that it has no universal covering, i.e. a simply connected covering space \cite{Massey2}. This fact affects the development of a potential theory on $K$. In an ordinary manifold $M$, any closed form $\omega$ has a pull back $\widetilde\omega$ on the universal covering space $\widetilde M$, which is obviously still closed but also exact, since $\widetilde M$ is simply connected. Hence, any closed form on a manifold admits a primitive function $U$ on $\widetilde M$, in the sense that $dU=\widetilde\omega$. Moreover, the primitive $U$ is a potential of $\o$ in the sense that its line integral along a path $\gamma$ in $M$ can be computed by the formula
\[
\int_\gamma\omega = U(p)-U(q)
\]
where $q,p\in \widetilde M$ are the initial and final points, respectively, of a lifting $\widetilde\gamma$ in $\widetilde M$ of $\gamma$.
\vskip0.2truecm\noindent
For the needs of a potential theory on the gasket $K$, the role played by the universal covering of a manifold, acted upon by its fundamental group, will be played  by the Uniform Universal Cover $\widetilde K$ introduced by Berestovskii and Plaut \cite{BerPla}, and more precisely by its abelian counterpart $\widetilde L$, acted upon by the first \v{C}ech homology group $\check{H}_1 (K)$, which is a projective limit of finitely generated abelian groups. In particular, the potentials $U$ of 1-forms on $K$ will be affine functions on $\LL$.

\subsection{Main results}

We now come to a closer look at our results. Our first step is the  definition of line integrals of the elements of  the bimodule of universal 1-forms $\uforme$ on the Dirichlet algebra $\cf$ along elementary paths in $K$, namely finite unions of consecutive oriented edges in $K$. Also, we consider a quadratic form $Q$ on $\uforme$ such that $Q[df]=\ce[f]$, as in the tangent bimodule construction. Now we have two natural quotients to take on  $\uforme$, either w.r.t. the intersection of the  kernels of the functionals $\o\mapsto\int_e\o$, where $e$ is any edge, or w.r.t. the kernel of the quadratic form $Q$. A main task will be to show that the kernels coincide, hence both the integrals and $Q$ make sense on the quotient. While the proof that $Q$ makes sense on the space $\forme$ of forms modulo forms with zero integral on edges is quite direct, the converse is not at all trivial. What we do is to analyze periods of forms $\o$ in $\forme$ around the lacunas of the gasket, and show that, given such periods, we may construct another form $\o'$ with the same periods in a canonical way as a series of a suitable sequence of forms $dz_\s$, parametrized by lacunas of $K$. We then prove that the difference $\o -\o'$ between the original form and the series is an exact form $dU$, thus showing at once the first and second de Rham theorems for the gasket, namely the fact that one may build a form given its periods, and the fact that such a form is indeed unique, up to exact forms.
In the same time, since the forms $dz_\s$ are harmonic, we obtain a Hodge theorem, i.e. we show that any form  has a harmonic representative in the space of (closed) forms modulo exact ones.  Finally, since the decomposition of a form $\o\in\forme$
\begin{equation*}
\o=d U+\sum_{\s}k_\s dz_\s
\end{equation*}
consists of pairwise orthogonal summands w.r.t. $Q$, we have that $Q[\o]=0$ implies $k_\s=0$ for all $\s$, and $\ce[U]=0$, namely $\o=0$, thus proving that $\forme$ coincides with the image of $\uforme$ in $\ch$ under the quotient map, hence is dense  in the tangent module $\ch$.
As a further outcome of our analysis, it turns out that the only natural definition of an external differential on 1-forms giving a differential complex is the trivial one, namely all 1-forms are closed,  in accordance with the fact that the gasket is topologically one-dimensional.

A second major issue of our paper is the attempt of extending the integral of a form from elementary paths to more general ones,  construct potentials of 1-forms, and prove a de Rham duality theorem. The space on which potentials of 1-forms will be defined is the projective limit $\widetilde L$ of a sequence of regular abelian covering spaces $\widetilde L_n$, where all loops around  lacunas of order up to $n$ are unfolded. Such pro-covering is acted upon by the \v{C}ech homology group $\check{H}_1(K,\bz)$ of the gasket, and is an abelian counterpart of the Uniform Universal Covering space introduced in \cite{BerPla}.
The results just mentioned will take two different versions, a purely algebraic one and a more analytical one.

The first version concerns locally exact forms. This subspace is the natural one from the point of view of algebraic topology, first because the integral of such forms extends naturally to all curves in the gasket; second, because any locally exact form $\o$ has a potential $U_\o$ on $\widetilde L$, such that the integral of $\o$ along a path $\g$ coincides with the variation of $U_\o$ at the end-points of a lifting of $\g$ to $\widetilde L$. Moreover, the potential $U_\o$ is  associated with a homomorphism $\f_\o:\check{H}_1(K,\br)\to\br$ such that $U_\o(g x)=U_\o(x)+\f_\o(g)$. The pairing $\langle\o,g\rangle=\f_\o(g)$  extends to a de Rham duality between $\check{H}_1(K,\br)$ and the space of locally exact forms modulo exact ones.

The second version concerns a suitable completion of closed smooth forms.
Indeed, in contrast with the classical situation, the space of locally exact smooth forms is a proper subspace of the space of closed smooth forms. Enlarging the class of forms  as to contain all smooth forms will correspond to a restriction of a class of allowed paths. 
We observe that smooth forms satisfy an estimate which puts them in a Banach space $\ch_N$ strictly contained in $\ch$. Then we prove that forms in $\ch_N$ may be integrated along all paths satisfying a dual estimate. Such paths are said to have finite effective length. Such finiteness can be read on the pro-covering as well. There, we define a (possibly infinite) distance $d$, which splits the space $\LL$ in $d$-components made of points with mutually finite distance, and selects a normal subgroup $\G_N$ of $\check{H}_1(K,\bz)$. Then, paths with finite effective length are those whose lifting lies in a single $d$-component, and homology classes of closed paths with finite effective length are elements of $\G_N$.
We then prove that any form in $\ch_N$ has a $\G_N$-affine potential  on any given $d$-component of  $\LL$, and  the integral of a form in $\ch_N$   along a path $\g$  with finite effective length coincides with the variation of the potential at the end points of a lifting of $\g$.  Moreover,  such integral gives  a nondegenerate pairing between $\G_N$ and the space $\ch_N$ modulo exact forms, more precisely the latter is the Banach space dual of  $\G_N\otimes_\bz\br$.

\section{The space of $1$-forms on the gasket}

\subsection{Preliminary notions}

We denote by $K$ the Sierpinski gasket, one of the most studied self-similar fractal sets. It was introduced in \cite{S} as a curve with a dense set of ramified points and has been the object of investigations in Probability \cite{Ku}, Theoretical Physics \cite{RT} and Mathetical Analysis \cite{FS,Kiga,St06}.
\par\noindent
Let $p_0 := (0,0)$, $p_1:= (\frac12,\frac{\sqrt3}2)$, $p_2:= (1,0)$ be the vertices of an equilateral triangle and consider the contractions $w_i$ of the plane: $x\in\br^2\to p_i+\frac12(x-p_i)\in\br^2$. Then $K$ is the unique fixed-point w.r.t. the contraction map $E \mapsto \cup_{i=0}^2 w_i(E)$ in the set of all compact subsets of $\br^2$, endowed with the Hausdorff metric. Two ways of approximating $K$ are shown in Figures \ref{Fig1} and \ref{Fig2}.
\vskip0.2truecm\noindent
Let us denote by $\Sigma_m :=\{0,1,2\}^{m}$ the set of words  composed by $m$ letters chosen in the alphabet  $\{0,1,2\}$, and by $\Sigma :=\bigcup_{m\ge 0}\Sigma_m$ the whole vocabulary (by definition $\Sigma_0 :=\{\emptyset\}$). A word $\sigma\in \Sigma_m$ has, by definition, length $m$, and this is denoted by $|\sigma |:=m$. For $\sigma =\s_1 \s_2\dots \s_m\in\Sigma_m$, let us denote by $w_\sigma$ the contraction $w_\sigma :=w_{\s_1}\circ w_{\s_2}\circ\dots \circ w_{\s_m}$.
\vskip0.2truecm\noindent
Let $V_0 :=\{p_0 ,p_1 ,p_2\}$ be the set of vertices of the equilateral triangle and $E_0 :=\{e_{0}, e_{1}, e_{2}\}$ the set of its edges, with $e_i$ opposite to $p_i$. Then, for any $m\ge 1$,  $V_m :=\bigcup_{|\sigma |=m} w_\sigma (V_0)$ is the set of vertices of a finite graph ($i.e.$ a one-dimensional simplex) $(V_m,E_m)$ whose edges are given by $E_m :=\bigcup_{|\sigma |=m} w_\sigma (E_0)$ (see Figure 2). The self-similar set $K$ can be reconstructed also as an Hausdorff limit either of the increasing sequence $V_m$ of vertices or of the increasing sequence $E_m$ of edges, of the above finite graphs. Set $V_* := \cup_{m=0}^\infty V_m$, and $E_* := \cup_{m=0}^\infty E_m$.
     \begin{figure}[ht]
 	 \centering
	 \psfig{file=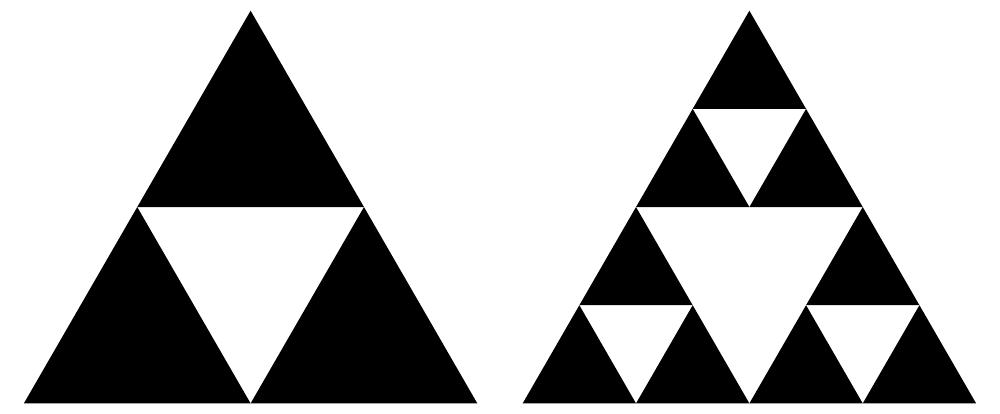,height=1.2in}
	 \psfig{file=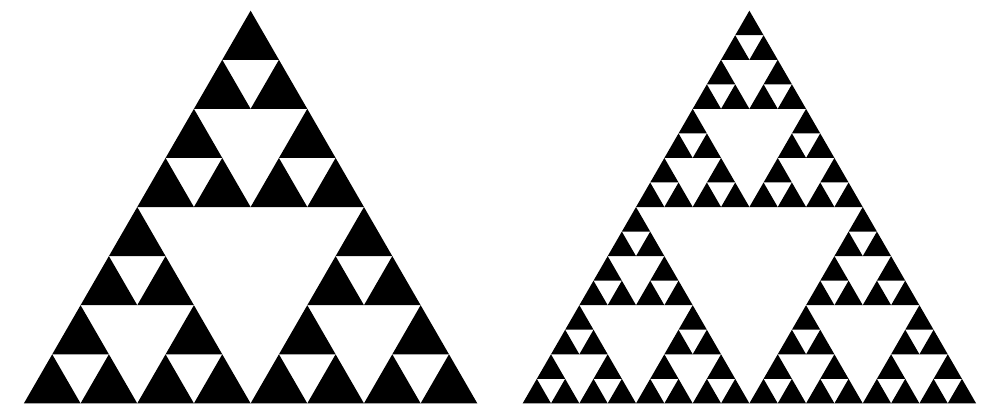,height=1.2in}
	 \caption{Approximations from above of the Sierpinski gasket.}
	 \label{Fig1}
     \end{figure}
     \begin{figure}[ht]
 	 \centering
	 \psfig{file=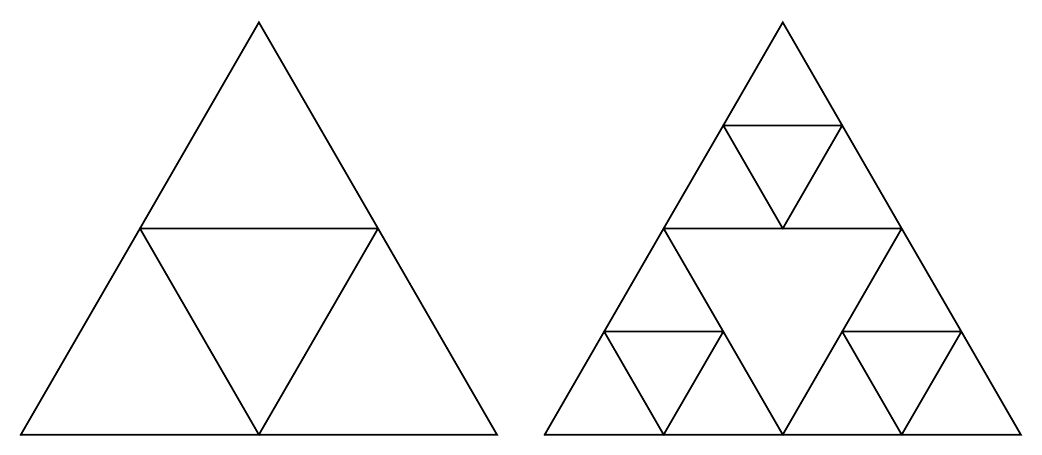,height=1.2in}
	 \psfig{file=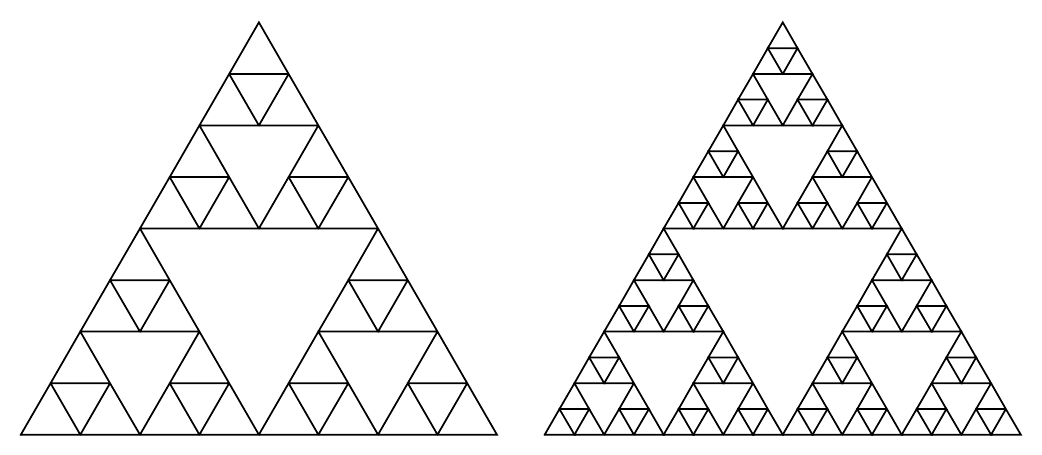,height=1.2in}
	 \caption{Approximations from below of the Sierpinski gasket.}
	 \label{Fig2}
     \end{figure}

In the present work a central role is played by the quadratic form $\ce : C(K)\to [0,+\infty ]$ given by
$$
\ce[f]=\lim_{m\to\infty} \left(\frac53\right)^m\sum_{e\in E_m}|f(e_+)-f(e_-)|^2,
$$
where each edge $e$ has been arbitrarily oriented, and $e_-,e_+$ denote its source and target. It is a regular Dirichlet form since it is lower semicontinuous, densely defined on the subspace  $\cf:= \set{f\in C(K) : \ce[f]<\infty}$ and satisfies the {\it Markovianity property}
\begin{equation}\label{contraction}
\ce [f\wedge 1]\leq \ce [f]\qquad f\in C(K)\, .
\end{equation}

The existence of the limit above and the mentioned properties are consequences of the theory of {\it harmonic structures} on self-similar sets developed by  Kigami \cite{Kiga}.
As a result of the theory of Dirichlet forms \cite{BD,FOT}, the domain $\cf$ is an involutive subalgebra of $C(K)$ and, for any fixed  $f,g\in \cf$,
the functional
\begin{equation}\label{rem:cs}
\cf\ni h\mapsto   \frac12\big(  \ce (f,gh)-\ce(f\overline{g},h)+\ce(\overline{h}f,g)
 \big)
\end{equation}
extends to a continuous functional on $C(K)$ so that it can be represented by a finite Radon measure called the {\it energy measure} (or {\it carr\'e du champ}) {\it of $f$ and $g$} and denoted by $\mu (f,g)$. In particular, for $f\in\cf$, $\mu (f,f)$ is a nonnegative measure and one has the representation
\[
\ce [f]=\int_K 1\, d\mu (f,f) =\mu (f,f)(K)\qquad f\in\cf\, .
\]
In applications, $f$ may represent a configuration of a system, $\ce [f]$ its corresponding total energy and $\mu (f,f)$ represents its energy distribution.
\par\noindent
In the present work we will denote with $\uforme$ the $\cf$-bimodule of  {\it universal $1$-forms} \cite{GVF} on $\cf$, that is $\uforme$ is the sub-$\cf$-bimodule of the algebraic tensor product $\cf \otimes \cf$, generated by elements of the form $fdg$, where the differential operator $d$ is defined by $df := f\otimes 1-1\otimes f$, $f\in\cf$ and the bimodule operations are given by $f\,dg = f(g\otimes 1-1\otimes g) := fg\otimes 1-f\otimes g$ and $dg\, f := d(gf)-g\, df = g\otimes f - 1 \otimes gf$, $f,g\in\cf$.
\par\noindent
As observed in  \cite{CiSa}, in a general regular Dirichlet space over a locally compact, separable Hausdorff space $X$, the properties of the Dirichlet form give rise to a positive semi-definite inner product on $\uforme$ given by the linear extension of the form
\begin{equation}\label{innerProduct}
Q(fdg,hdk)=\int_X\overline{f}h\,d\mu(g,k)\qquad f,g,h,k\in\cf\, .
\end{equation}
By separation and completion, this gives rise to a Hilbert space $\ch$ which is in fact a Hilbert $C_0(X)$-bimodule called the {\it tangent bimodule associated to $\ce$} and whose elements are called {\it square integrable forms}. In the present case of the Sierpinski gasket, since the Dirichlet form is strongly local, the right and left actions coincide so that $\ch$ is a Hilbert $C(K)$-module. Moreover, the derivation $\partial :\cf\to\ch$, associated to the Dirichlet form $\ce$ comes directly from the universal derivation $d:\cf\to \uforme$ in such a way that
\begin{equation}\label{innerProduct1}
Q(fdg,hdk)=
(f\partial g ,h\partial k)_\ch\, ,
\qquad
Q[df]=\|\partial f\|^2_\ch = \ce[f]\qquad f,g,h,k\in\cf\, .
\end{equation}

\medskip

The Dirichlet or energy form $\ce$ should be considered as a Dirichlet integral on the gasket. It is closable with respect to any Borel regular probability measure on $K$ which is positive on open sets and vanishes on finite sets (see \cite{Kiga} Theorem 3.4.6 and \cite{Kiga3} Theorem 2.6). Once such a measure $m$ has been chosen, $\ce$  gives rise to a positive, self-adjoint operator on $L^2 (K, m)$, which may be thought of as a Laplace-Beltrami operator on $K$. However, since in the present work only the Dirichlet form   plays a role, we content ourselves with  the measure-valued Laplacian, as studied in \cite{Kiga2}.

A function $f\in \cf$ is said to be {\it harmonic} in a open set $A\subset K$ if, for any $g\in\cf$ vanishing on $A^c$, one has
\[
\ce (f,g)=0\, .
\]
As a consequence of the Markovianity property \eqref{contraction}, a Maximum Principle holds true for harmonic functions on the gasket \cite{Kiga}.
In particular, one calls $0$-{\it harmonic} a function $u$ on $K$ which is harmonic in $V_0^c$. Equivalently, for given boundary values on $V_0$, $u$ is the unique function in $\cf$ such that $\ce[u] = \min\set{\ce[v]: v\in\cf, v|_{V_0} = u}$. More generally, one may call  $m$-{\it harmonic} a  function that,  given its values on $V_m$, minimizes the energy among all functions in $\cf$. For such functions we have
$$
\ce[u]=\left(\frac53\right)^m\sum_{e\in E_m}|u(e_+)-u(e_-)|^2\, .
$$
It is not difficult to check that $f\in \cf$ is $m$-harmonic if and only if $\Delta f$ is a linear combination of Dirac measures supported on the vertices $V_m$.

\begin{dfn} (Cells, lacunas)
For any word $\s\in\Sigma_m$, define a corresponding {\it cell} in $K$ as follows
\[
C_\s=w_{\s}(K)\, ,
\]
its {\it perimeter} by $\perim C_\s =w_{\s}(E_0)$, its (combinatorial) {\it boundary} by $\bordo C_\s =w_{\s}(V_0)$ and its (combinatorial) {\it interior} by $C_\s^o =C_\s\setminus \bordo C_\s$. We will also define the {\it lacuna} $\ell_\emptyset$, see Fig.~\ref{lacuna}, as the topological boundary of the first removed triangle according to the approximation in Fig.~\ref{Fig1}. For any $\s\in\Sigma$, the lacuna $\ell_\s$ is defined as $\ell_\s =w_{\s}(\ell_\emptyset)$.
     \begin{figure}[ht]
 	 \centering
	 \psfig{file=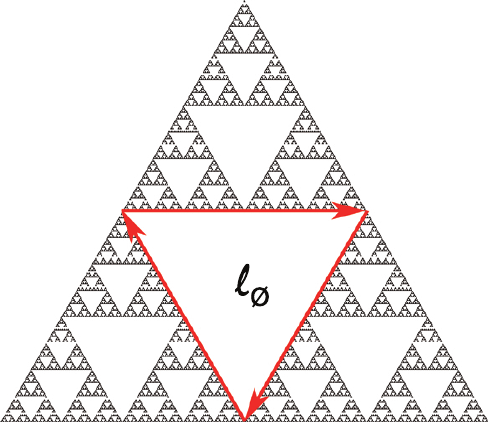
	 }
	  \caption{The lacuna $\ell_\emptyset$}
	 \label{lacuna}
     \end{figure}
\\
For a function $f$ on $K$, let us define its {\it oscillation on a closed subset $T\subseteq K$} as
\[
 \osc(f)(T):=\max_{x,y\in T}|f(x)-f(y)|=\max_{T}f-\min_{T}f\, .
\]
\end{dfn}

It is easy to check that if $f$ is a harmonic function in the interior of a cell $C$ and $C_1$ is one of its three sub-cells, then $\osc(f)(C_1)\leq\frac35\osc(f)(C)$ (see for example \cite{St06} Chapter 1 Exercise 1.3.6).

Recall that $\uforme$ denotes the bimodule of universal 1-forms over the algebra $\cf$.
There is a natural pairing between elements of $\cf \otimes \cf$ and oriented edges which is given  by $(f\otimes g)(e):=f(e_+)g(e_-)$ on elementary tensors. As a consequence,
\begin{align}
d g (e)&=g(e_+)-g(e_-)\\
(f\, d g)(e)&=f(e_+)d g(e)\label{pairing1}\\
(d g\, f)(e)&=f(e_-)d g(e).\label{pairing2}
\end{align}

\subsection{Integrating $1$-forms along elementary paths}\label{subs:Integral}

\begin{dfn}
A path in $K$ given by a finite union of consecutive oriented edges in $E_*$ is called {\it elementary}.
\end{dfn}

\noindent
Let $\gamma$ be an oriented elementary path in $K$ and $\o = \sum_{i\in I} f_idg_i \in\uforme$. For $n\in\bn$, define
$$I_n(\g)(\o) = \sum_{e\in E_n(\gamma)} \o(e),$$
where  $E_n(\gamma)$ denotes the set of oriented edges of level $n$ contained in $\g$.

\begin{dfn}\label{dfn:integral}
We define the integral of a 1-form $\o$ along an elementary path $\g$ as the limit $\int_\gamma\o = \lim_{n\to\infty} I_n(\g)(\o)$.
\end{dfn}
\begin{rem}
The  integral of 1-forms defined above is a kind of Riemann-Stieltjes integral conditioned to diadic partitions of edges. Unfortunately, while the classical result of Young \cite{Young} for $\int fdg$ requires H\"older continuity of $f$ and $g$ with sum of the exponents $>1$, restrictions to edges of finite energy functions on the gasket are known to be only $\b$-H\"older, with $\b<1/2$  (cf. e.g. \cite{Jo96}), therefore we cannot use Young result. Also, restrictions to edges of finite energy functions are not of bounded variation in general\footnote{In \cite{ACSY}, p.18, examples are given of finite energy functions with non BV restriction to edges, but is observed that harmonic functions do have BV restriction to edges. As a consequence, the integral of the form in Proposition \ref{Prop:NonLocEx} along an elementary path makes sense as a Lebesgue-Stieltjes integral.}, therefore we cannot use Lebesgue-Stieltjes integral either.
Nevertheless, on identifying an edge $e\in E_*$ with $[0,1]$, the bilinear form $(f,Dg)_e$ on $L^2(e)$ given by $\int_0^1 f(x)g'(x)\,dx$ for $f,g$ smooth functions, naturally extends to a bounded form on  $H^{1/2}(e)$, hence makes sense also for $f,g\in\cf$ since, by results of Jonsson \cite{Jo}, traces of finite energy functions on edges $e\in  E_*$ belong to the fractional Sobolev space $H^{\a} (e)$ for $\a\leq\a_0$, $\a_0=\frac{\log(10/3)}{\log4}\sim0.87$.
The existence of the limit in Definition \ref{dfn:integral} and the coincidence of the two notions are proved below.
\end{rem}

\begin{thm} \label{JLnotes}
Let $\o\in\uforme$ be a  1-form and $\g$  an elementary path in $K$. Then
\begin{itemize}
\item[$(i)$] the integral $\int_\gamma\o$ is well defined,
\item[$(ii)$] the integral is a bimodule trace, namely
\[
\int_\gamma h\, \o = \int_\gamma \o\, h\qquad h\in\cf\, ,
\]
\item[$(iii)$] for all $h\in\cf$, the following approximation holds true:
\[
\int_\gamma h\, \o=\lim_n\sum_{e\in E_n(\g)}h(e_+)\int_e\o\, .
\]
\item[$(iv)$] Let $e$ be an edge in $K$, $f,g$  finite energy functions on $K$. Then
\begin{equation}\label{twointegrals}
\int_e fdg=(f,Dg)_{e}.
\end{equation}
\end{itemize}
\end{thm}
\begin{proof}
It is not restrictive to assume $\o=fdg$. We choose $n_0\in\bn$ such that $\gamma$ is a finite union of edges of level $n_0$.
\\
$(i)$ For $n\geq n_0$ and $e\in E_n(\g)$, let $e^0\in V_{n+1}$ be the middle point of the edge $e$. One computes
\begin{align}
I_{n+1}(fd g)&=\sum_{e\in E_n(\gamma)}f(e_+)\big(g(e_+)-g(e^0)\big) + \sum_{e\in E_n(\gamma)} f(e^0)\big(g(e^0)- g(e_-)\big) \label{eq:diff}\\
&=I_n(fd g)+ \sum_{e\in E_n(\gamma)} \big(f(e^0)-f(e_+)\big)\big(g(e^0)-g(e_-)\big)\,,\notag
\end{align}
so that
\begin{align}
\left| I_{n+1}(fd g)-I_n(fd g)\right| & \leq \left( \sum_{e\in E_n( \gamma )}
\left| f(e^0)-f(e_+)\right|^2 \right)^{1/2}\, \left( \sum_{e\in E_n(\gamma)}
\left| g(e^0)-g(e_-)\right|^2 \right)^{1/2}\notag\\
&\leq \left( \sum_{e\in E_{n+1}(\gamma)} |d f(e)|^2\right)^{1/2}\,  \left( \sum_{e\in E_{n+1}(\gamma)} |d g(e)|^2\right)^{1/2}\label{eq:a}\\
&\leq \frac12  \sum_{e\in E_{n+1}(\gamma)} (|d f(e)|^2+|d g(e)|^2)\label{eq:b}\\
&\leq \frac12\left(\frac35\right)^{n+1}\,( \ce[f]+\ce[g])\,.\label{eq:c}
\end{align}
Hence,
$$
| I_n(\g)(fdg)- I_{n+p}(\g)(fdg) |
\leq \sum_{k=n}^{n+p-1}I_{k+1}(fd g)-I_k(fd g) |
\leq\frac34 (\ce[f]+\ce[g])\left(\frac35\right)^{n},
$$
namely the sequence $I_n(\g)(fdg)$ converges.
\\
$(ii)$ The result follows form
$$
 I_n(\g)(h\,fdg) - I_n(\g)(fdg\,h)|
\leq \|f\|_\infty \sum_{e\in E_n( \gamma)} | dh(e)| \, | dg(e)|
\leq \frac12 \|f\|_\infty (\ce[h]+\ce[g])\Big( \frac35 \Big)^n .
$$
$(iii)$ The thesis follows from
\begin{align*}
\bigg| I_n(\g)(h\o)-&\sum_{e\in E_n(\g)}h(e_+)\int_e\o\bigg|
\leq \sum_{e\in E_n(\g)}|h(e_+)|\ \left| \o(e)-\int_e\o\right|\\
&\leq \|h\|_\infty\sum_{e\in E_n(\g)}\sum_{p=0}^\infty | I_{p+n+1}(e)(fd g)-I_{p+n}(e)(fd g)|\\
&\leq \frac12\|h\|_\infty\sum_{p=0}^\infty\sum_{e\in E_n(\g)}
\sum_{e'\in E_{p+n+1}(e)} (|d f(e')|^2+|d g(e')|^2)\\
&\leq \frac12\|h\|_\infty\sum_{p=0}^\infty\sum_{e'\in E_{p+n+1}(\g)} (|d f(e')|^2+|d g(e')|^2)\\
&\leq \frac12\|h\|_\infty( \ce[f]+\ce[g])\sum_{p=0}^\infty\left(\frac35\right)^{p+n+1}
\leq \frac34\|h\|_\infty( \ce[f]+\ce[g])\left(\frac35\right)^{n}.
\end{align*}
$(iv)$ Given a function $f$ on an edge $e$, consider the  continuous piecewise-linear approximation $f_n$  which coincides with $f$ on diadic points of  $e$ identified with the interval $[0,1]$:
$$
f_n(x)=\sum_{j=1}^{2^n}\chi_{[(j-1)2^{-n},j2^{-n})}(x)\Big( f((j-1)2^{-n})+\frac{f(j2^{-n})-f((j-1)2^{-n})}{2^{-n}} (x-(j-1)2^{-n})\Big).
$$
Since eq. \eqref{twointegrals} clearly holds for continuous piecewise-linear functions, it is sufficient to show that both terms in \eqref{twointegrals} are continuous w.r.t. the approximation above. By definition,
$I_k(fdg)=I_k(f_ndg_n)$, $n\geq k$, therefore
$$
\bigg| \int_e fdg-\int_e f_ndg_n \bigg| \leq \bigg| \int_e fdg-I_n(fdg) \bigg|+ \bigg| I_n(f_ndg_n)-\int_e f_ndg_n \bigg| \to0,
$$
since the first summand goes to 0 by the preceding Theorem \ref{JLnotes}, and, setting $|e|=p$,
$$
\bigg| I_n(f_ndg_n)-\int_e f_ndg_n \bigg|
=
\sum_{e'\in E_{p+n}(e)} d f(e')\ d g(e')
\leq\frac12 \left(\frac35\right)^{n+p}( \ce[f]+\ce[g]).
$$
As for the bilinear form, it is sufficient to show that $f_n\to f$ in $H^{1/2}(e)$.
According to
\cite{Jo}, a norm for the Sobolev spaces $H^\a[0,1]$, $1/2<\a<1$, is
$$
\|f\|_{H^\a}=(f(0)^2+f(1)^2)^{1/2}+\Big(\sum_{n=0}^\infty2^{n(2\a-1)}E_n(f)\Big)^{1/2},
$$
where
$$
E_n(f)=\sum_{j=1}^{2^n}\big(f(j2^{-n})-f((j-1)2^{-n})\big)^2.
$$
Therefore,
$$
\|f-f_k\|^2_{H^\a}=\sum_{n=k+1}^\infty2^{n(2\a-1)}E_n(f-f_k)
\leq2\sum_{n=k+1}^\infty2^{n(2\a-1)}E_n(f)+2\sum_{n=k+1}^\infty2^{n(2\a-1)}E_n(f_k).
$$
If $\a\leq\a_0$, the first summand is a remainder of a convergent series, hence goes to $0$, as $k\to\infty$. As for the second, since $f_k$ has constant slope on diadic  intervals of length $2^{-k}$, a direct computation shows that, for $n>k$, $\displaystyle E_n(f_k)=2^{k-n}E_k(f)$,
herefore
$$
\sum_{n=k+1}^\infty2^{n(2\a-1)}E_n(f_k)= (2^{2-2\a}-1)^{-1}2^{k(2\a-1)}E_k(f)\to0
$$
since $2^{k(2\a-1)}E_k(f)$ is the generic term of a convergent series. This shows that, for $\a\in(1/2,\a_0]$, $f_k\to f$ in $H^{\a}([0,1])$. The convergence in $H^{1/2}[0,1]$ then follows.
\end{proof}

Our aim now is to show that the integral defined above on $\uforme$ makes sense on (sufficiently regular) elements of the tangent module, namely to show that the integral passes to the quotient w.r.t. forms in the kernel of the quadratic form $Q$ described in eq. \eqref{innerProduct}.
However, in order to achieve this result, we have to take a different quotient, namely to identify forms whose integrals coincide on any edge, and to dwell in such space for a while. After proving a series of results, which have an interest in their own, we will be able to prove that the latter quotient indeed coincides with the former, i.e.
\begin{equation}\label{equivalence}
\int_e\o=0\ \forall e\in E_*\Longleftrightarrow Q[\o]=0.
\end{equation}

\begin{dfn}
Let us now introduce the equivalence relation on $\uforme$ given by $\o \sim \o' \iff \int_e (\o-\o')=0$, for all $e\in E_*$, and consider the quotient space $\forme := \uforme/\sim$. We call {\it smooth 1-forms} the elements of $\forme$.
\end{dfn}
In the following, we use the shorthand notation $\ce_C[f] := \ce[f|_C]$, for any cell $C$ in $K$.

\begin{lem}\label{Q-norm}
For any $\o\in\uforme$,
 \begin{equation}\label{QviaIntegrals}
Q[\o]= \lim_{n\to\infty} (5/3)^n\sum_{e\in E_n} \bigg| \int_e\o \bigg|^2\, .
\end{equation}
As a consequence, the quadratic form $Q$ is well defined on the space $\forme$.
\end{lem}
\begin{proof}
Let us set
\begin{equation}\label{nquadraticform}
Q_n[\o]=  (5/3)^n\sum_{e\in E_n}\bigg| \int_e\o \bigg|^2,\qquad \widetilde Q_n[\o]=(5/3)^n\sum_{e\in E_n}|\o(e)|^2\, ,\qquad \o\in\uforme\, .
\end{equation}

We have
$$
\widetilde Q_n[f\,dg-dg\,f]=\left(\frac53\right)^{n}\sum_{e\in E_n}df(e)^2dg(e)^2
\leq\ce_n[f]\max_{e\in E_n}dg(e)^2,
$$
hence $\lim_n \widetilde Q_n[f\,dg-dg\,f]=0$.
A straightforward computation gives
$$
\widetilde Q_n(dg,f\,dh)+\widetilde Q_n(dg,dh\,f)=\ce_n(g,fh)-\ce_n(gh,f)+\ce_n(h,fg),
$$
therefore
\begin{align*}
\widetilde Q_n(dg,f\,dh)
&=\frac12\left(\widetilde Q_n(dg,f\,dh)+\widetilde Q_n(dg,dh\,f)+\widetilde Q_n(dg,f\,dh-dh\,f)\right)\\
&=\frac12\big(\ce_n(g,fh)-\ce_n(gh,f)+\ce_n(h,fg)\big)+\frac12\widetilde Q_n(dg,f\,dh-dh\,f)\\
&\to \frac12\Big(    \ce(g,fh)-\ce(gh,f)+\ce(h,fg)   \Big)
=\int_K f\,d\mu(g,h),
\end{align*}
therefore $\widetilde Q_n\to Q$. We finally prove that the two limits $\lim_{n\to\infty} Q_n[\o]$, $ \lim_{n\to\infty}\widetilde Q_n[\o]$ coincide.
For sequences $x=\{x_e:e\in E_*\}$, we introduce the seminorms
\begin{equation}\label{phin}
\Phi_n(x) := \left( \frac53 \right)^{n/2} \left( \sum_{e\in E_n} |x_e|^2 \right)^{1/2}.
\end{equation}
In particular, $\widetilde Q_n[\o]=\Phi_n(\o(e))^2$ and $Q_n[\o]=\Phi_n(\int_e\o)^2$.
Let us  denote with $C(e)$ the cell having $e$ as one of its boundary segments. We get, by inequality (\ref{eq:a}),
\begin{align*}
\Phi_n\left( (f_idg_i)(e) - \int_ef_idg_i \right)^2
& = \Big( \frac53 \Big)^n \sum_{e\in E_n} \big| I_n(e)(f_idg_i) - \lim_{k\to\infty} I_k(e)(f_idg_i) \big|^2 \\
& \leq \Big( \frac53 \Big)^n \sum_{e\in E_n} \Bigg(    \sum_{j=n}^\infty \big| I_{j+1}(e)(f_idg_i) - I_j(e)(f_idg_i) \big| \Bigg)^2 \\
& \leq \Big( \frac53 \Big)^n \sum_{e\in E_n} \Bigg(    \sum_{j=n}^\infty \Big( \frac35\Big)^{j+1} \ce_{C(e)}[f_i]^{1/2} \ce_{C(e)}[g_i]^{1/2} \Bigg)^2 \\
& = \frac94 \Big( \frac35 \Big)^n \sum_{e\in E_n}   \ce_{C(e)}[f_i] \   \ce_{C(e)}[g_i]
 \leq \frac{27}4 \Big( \frac35 \Big)^n  \ce[f_i]     \ce[g_i] .
\end{align*}
As a consequence, for $\o=\sum_{i\in I}f_idg_i$,
\begin{align*}
\big| \widetilde Q_n[\o]^{1/2}-Q_n[\o]^{1/2} \big|
&=\bigg| \Phi_n(\o(e)) - \Phi_n(\int_e\o) \bigg|
\leq\bigg| \Phi_n(\o(e) - \int_e\o) \bigg|\\
&\leq\sum_{i\in I}\bigg| \Phi_n((f_idg_i)(e) - \int_ef_idg_i) \bigg|
\leq\frac{3\sqrt3}2 \Big( \frac35 \Big)^{n/2}\sum_{i\in I}  \ce[f_i] ^{1/2}    \ce[g_i]^{1/2}.
\end{align*}
\end{proof}

The following Proposition summarizes the previous results.

\begin{prop}
The space $\forme$ is an $\cf$-module and the universal derivation gives rise to a derivation $d:\cf\to\forme$ (still indicated by the same symbol).
The integral along an elementary path and the seminorm $\normaQ$ are well defined on $\forme$.
\end{prop}

\subsection{Locally exact 1-forms}
On a smooth manifold $M$, a closed form $\o$ is locally exact, namely $\forall x\in M$, there exists a pair $(\cv_x,f_x)$, where $\cv_x$ is a neighborhood of $x$ and  $f_x$ is a local potential of $\o$ on $\cv_x$, that is  $f_x$ satisfies
\begin{equation}\label{LocInt}
\int_\g\o=f_x(\g(1))-f_x(\g(0)),\qquad \forall \g\subset \cv_x.
\end{equation}
A family of local potentials as above may be abstractly described as a pair $(\{U_i\},\{f_i\})$ where  $\{U_i\}$ is an open cover of $M$ and $f_i$ is a continuous function on $U_i$ such that $(f_i-f_j)|_{U_i\cap U_j}$ is locally constant. Clearly such pairs can be considered for any topological space\footnote{On a smooth manifold, codimension-1 smooth  foliations are in 1:1 correspondence with closed 1-forms, the longitudinal tangent of the former being locally described as the kernel of the 1-form, or, equivalently, as the level sets of the local potentials of the 1-form.
The latter description extends to topological spaces,  giving rise to codimension-1 $C^0$-foliations and coincides with the pairs $(\{U_i\},\{f_i\})$ considered above \cite{CaNe}.}.

We say that $(\{U_i\},\{f_i\})$ is equivalent to $(\{V_i\},\{g_i\})$ if $(f_i-g_j)|_{U_i\cap V_j}$ is locally constant, denote the quotient space by $\toplex$ and call its elements
 {\it locally exact topological 1-forms}.
As shown below, the integral in \eqref{LocInt} extends to a pairing  between locally exact  topological 1-forms and continuous paths in $X$.

\begin{lem} \label{}
	Let $X$ be a topological space,  $( \{U_i\}_{i\in I}, \{f_i\}_{i\in I})$ a representative of a locally exact topological  1-form $\o$ as above, and $\g:[0,1]\to X$ a continuous path. Then $\int_\g \o$ is well defined.
\end{lem}
\begin{proof}
	The family $\set{\g^{-1}(U_i):i\in I}$ is an open cover of $[0,1]$, so we can consider its Lebesgue number $\d>0$. Let $\set{t_0=0,t_1,\ldots,t_n=1}$ be a partition of $[0,1]$ such that $t_k-t_{k-1}<\d$, $k=1,\ldots,n$, so that, for any $k$, $\g([t_{k-1},t_k]) \subset U_{i_k}$ for some $i_k\in I$. Then we define $\int_\g\o:= \sum_{k=1}^n f_{i_k}(\g(t_k))-f_{i_k}(\g(t_{k-1}))$. Suppose now that
$(\{V_i\},\{g_i\})$ is another representative of $\o$ with Lebesgue number $\d'>0$, $\set{t'_0=0,t_1,\ldots,t'_m=1}$ the corresponding partition of $[0,1]$ such that $t'_j-t'_{j-1}<\d'$, $j=1,\ldots,m$, and denote by $\set{s_0=0,s_1,\ldots,s_\ell=1}$ the union of the two partitions. Clearly the two integrals coincide, proving the statement.
\end{proof}

We now show that, when $f_i$'s have finite energy, any $(\{U_i\},\{f_i\})$ gives rise to a unique element of $\forme$.
Such elements will be called locally exact smooth 1-forms, the corresponding space will be denoted by $\qex$.
We first note that, because of the topology of $K$, any $(\{U_i\},\{f_i\})$ may be equivalently represented by $\{f_\s\}_{|\s|=n}$ for some $n$, where $f_\s$ is a local potential on the closed cell $C_\s$.

\begin{prop}\label{omega0norm}
Let $\{f_\s\}_{|\s|=n}$ be a family of local potentials as above representing a locally exact $1$-form $\o_0$, with $\ce_{C_\s}[f_\s]<\infty$ for any $\s$. Then there exists a unique $1$-form $\o\in\forme$ such that $\int_\g\o=\int_\g\o_0$ for any elementary path $\g$. Such $\o$ will be called $n$-{\it exact}.
Moreover, $Q[\o]=\sum_{|\s|=k} \ce_{C_\s}[f_\s]$ and $\normaQ$ is a norm on the space $\qex$ of locally exact forms with finite-energy local potentials.
\end{prop}

\begin{proof}
We may  associate with any $f_\s$ in the statement an element in $\forme$ as follows: let $A\supset C_\s$ be an open set in $K$ such that the connected components of $(K \setminus C_\s^o)\cap \ov A$ are cells containing exactly one boundary vertex of $C_\s$; let $\tilde f_\s$ be a function in $\cf$ which coincides with $f_\s$ in $C_\s$ and is constant on each connected component of $(K \setminus C_\s^o)\cap \ov A$;  and let $\chi_\s$ be a function in  $\cf$ which is 1 on $C_\s$ and has support contained in $A$. If we set $\o_\s=\chi_\s d\tilde{f}_\s$,
then
$$
\int_e \o_\s = \lim_{n\to \infty} \sum_{\substack{e'\in E_n \\ e'\subset e}} \chi_\s(e'_+) ( \tilde f_\s(e'_+)- \tilde f_\s(e'_-) ).
$$
Now, if $e$ intersects $C_\s$ at most in one vertex, we get $\int_e \o_\s = 0$, because $\tilde f_\s$ is constant on any $e'\in E_n$, $e'\subset e$. If, on the contrary, $e\subset C_\s$, then $\chi_\s(e'_+)=1$, for any such $e'$, while $\tilde f_\s = f_\s$, so that $\int_e \o_\s = \lim_{n\to \infty} \sum_{\substack{e'\in E_n \\ e'\subset e}}  (  f_\s(e'_+)-  f_\s(e'_-) ) = f_\s(e_+)-f_\s(e_-)$.
Clearly $\sum_{|\s|=n}\o_\s$ is the required $n$-exact form.
We now prove the second statement.
For any $n>k$, we get
\begin{align*}
Q_n(\o) & = \Big( \frac53 \Big)^n \sum_{e\in E_n} \Big| \int_e\o\Big|^2
=  \Big( \frac53 \Big)^n \sum_{|\t|=k} \sum_{e\in E_n (C_\t)} \Big| \int_e\o\Big|^2
 = \sum_{|\t|=k} \ce_n[ f_\t].
\end{align*}
Therefore, $\ds Q(\o) = \lim_{n\to\infty} Q_n(\o) = \sum_{|\s|=k} \lim_{n\to\infty} \ce_n[ f_\s] = \sum_{|\s|=k} \ce[f_\s]$.
Finally, $0 = Q(\o) =  \sum_{|\s|=k} \ce[f_\s] \implies f_\s$ is constant on $C_\s$, for any $\s \implies \o=0$.
\end{proof}

We now introduce a distinguished system of locally exact smooth 1-forms associated with lacunas which will play a fundamental role in the following.

\begin{dfn}
For any $n\ge 0$ and $|\s|=n$, define  $dz_\s$ as the $(n+1)$-exact form which minimizes the norm $\normaQ$ among those $(n+1)$-exact 1-forms $\o$ satisfying  $\int_{\ell_\s} \o=1$.
\end{dfn}

By definition, $dz_\s$ is exact in any of the cells $C_{\s i}$, hence  $\int_{\perim C_\s}dz_\s=-1$ (lacunas are traversed clockwise and perimeters of cells are traversed counter-clockwise, according to the standard convention, as they constitute the boundary of the union of the convex hulls of the cells $C_{\s i}$, $i=1,2,3$). The minimization request implies that $dz_\s$ vanishes in any cell $C_\t$ with $\t\ne\s$, $|\t|=n$, and that $dz_\s$ is symmetric for rotations of $\frac23\pi$ around $\ell_\s$.

\begin{prop}\label{wclosed}
$(i)$ The forms $dz_\s$ are weakly co-closed, i.e. orthogonal to all exact smooth 1-forms, and pairwise orthogonal, with
\begin{equation}\label{dzsigmanorm}
Q[ dz_\s]=\frac56\left(\frac53\right)^{|\s|}.
\end{equation}
$(ii)$ Any $n$-exact topological form has a unique decomposition as the sum of an exact topological form plus a finite linear combination of $dz_\t$, $|\t|<n$. If the form is smooth, the decomposition is  orthogonal w.r.t. the quadratic form $Q$. Uniqueness of the decomposition implies that $Q^{1/2}$ is a norm on locally exact smooth forms.
\end{prop}
\begin{proof}
$(i)$ A simple calculation shows that for any cell $C_{\s i}$, the local potential $z_\s^i$ on such cell is the harmonic function determined (up to an additive constant) by the values $\frac16,0,-\frac16$ on the vertices $x_1,x_2,x_3$, where $x_3,x_1$ is the edge bounding the lacuna. Therefore, $\D z^i_\s$ may be canonically identified with the measure given by the linear combination $\frac12\d_{x_1}-\frac12\d_{x_3}$. As a consequence, for any $f\in\cf$,
$$
Q(df,dz_\s)=\sum_{i=1,2,3}Q(df,dz^i_\s)=\sum_{i=1,2,3}\ce(f,z^i_\s)=\sum_{i=1,2,3}\int_K f d(\D z^i_\s)=0.
$$
If $\t<\s$ the orthogonality follows as above; if $\t$ and $\s$ are not ordered, $dz_\s$ and $dz_\t$ have disjoint support. The value of the norm follows from a direct computation.
\\
$(ii)$ We  note that, for any cell $C_\sigma$, an $(n+1)$-exact form $\omega$ on  $C_\sigma$ is indeed $n$-exact if and only if $\int_{\ell_\sigma}\omega=0$, since in this case the three local potentials on the three sub-cells may  glue to a continuous function on $C_\sigma$. Therefore, any $(n+1)$-exact form $\omega$ supported in $C_\sigma$ may be written as
$$
\omega=\left(\omega- c_\s  dz_\s\right)+c_\s  dz_\s,\qquad c_\s := \int_{\ell_\sigma}\omega,
$$
namely, for any cell $C_\sigma$, the codimension of $n$-exact forms into $(n+1)$-exact forms supported in $C_\sigma$ is 1. This shows that exact forms and the $dz_\t$, $|\t|<n$, generate the $n$-exact forms, hence the thesis.
When the form is smooth, the exact part in the decomposition is also smooth, hence the statement follows by Proposition \ref{wclosed}.
\end{proof}

Similarly to the case of an ordinary  smooth manifold, 1-forms which are locally exact and co-closed will be termed {\it harmonic}, therefore $\{dz_\s :\s\in\Sigma\}$ is an orthogonal system of harmonic 1-forms. A more general result is contained in Lemma \ref{DifferentialOnOmega1}.

\subsection{Winding numbers and a combinatoric way to describe lacunas bounding cells.}\label{sec:winding}

Since $dz_\s$ is invariant under rotations of $\frac23\pi$ around the lacuna $\ell_\s$, the integral along any edge $e$ bounding $C_\s$  is equal to $-1/3$. We now consider the integral $B_{\r\t}=\int_{\ell_\t}d z_\r$. It is not difficult to see that $B_{\r\t}$ does not vanish only if $\t\leq\r$ ($\t$ is a truncation of $\r$), more precisely,
$$
B_{\r\t}=\int_{\ell_\t}d z_\r=
\begin{cases}
1&\text{if}\ \t=\r,\\
-1/3 &\text{if}\ \ell_\t\cap \perim C_\r\neq\emptyset,\\
0&\text{otherwise}.
\end{cases}
$$
In particular, $B=\{B_{\r\t}\}$ is a lower unitriangular matrix with indices in $\Sigma$ (i.e. $B_{\r\t}=0$ for $\t>\r$ and $B_{\r\r}=1$ for any $\r$).
The following result is well known for finite matrices and for infinite matrices with indices in $\bz$ \cite{Hol}, but extends to infinite matrices with indices in a partially ordered set $\S$ such that $\{\t\in\S:\t\leq\s\}$ is finite and linearly ordered, cf. also \cite{DroGoe} for the case of finitary matrices.
\begin{prop}
The set $UT(\S,\br)$ of $\br$-valued lower unitriangular matrices with indices in $\Sigma$ is a group contained in $Aut(\br^\S)$.
\end{prop}
Let us observe that the product and the action on $\br^\Sigma$ are defined in a purely algebraic sense, since the sums involved are always finite.
Setting $A_{\s\r}=(B^{-1})_{\s\r}$, we get
\begin{equation}\label{asigmatau}
\int_{\ell_\tau}\sum_{\r\leq\s} A_{\s\r} dz_\r=\d_{\s\t}\qquad\s\, ,\t\in\Sigma\, .
\end{equation}

\begin{rem}[Winding number]
In other words, the 1-form $\o^\s :=\sum_{\r\leq\s} A_{\s\r} dz_\r$ detects only the lacuna $\ell_\s$.
As a consequence, for any closed path $\g$ in $K$,
\[
\int_\g \o^\s
 \]
is the {\it winding number of the path $\g$ around the lacuna $\ell_\s$}.
\end{rem}

\begin{lem}\label{lem:asigmatau}
With the notation above, $0\leq A_{\s\t}\leq1$, $\t\leq\s$.
\end{lem}
\begin{proof}
Since $A=B^{-1}$, we have $A^*=(1-B^*)A^*+1$, hence, setting $D=3(1-B^*)$, we get
\begin{equation}
A_{\s\t}=\frac13\sum_{\t\leq\r\leq\s} D_{\t \r}A_{\s\r}+\d_{\s\t}.
\end{equation}
For a given $\s$, let us rename indices and variables as follows: replace the $n$-th truncation $\s^{(n)}$ of $\s$ with $n$, so that the order is reversed, and rename $A_{\s\s^{(n)}}$ as $v_n$. Then the equation above becomes
$$
v_p=\frac13\sum_{j=0}^{p} D_{pj}v_j+\d_{0p}, \ p=0,1,\ldots,|\s|.
$$
Denoting by $P$ the projection on the $0$-th component, we get $v=(\frac13 D+P)v$. Recall that $D_{ij}$ may be non zero for at most three indices $i$ following $j$, and observe that $D$ is a lower triangular matrix, hence $(D^p)_{jk}$ does not vanish only if $k\leq j-p$, and $PD=0$. Therefore we get
$$
v=\Big(\frac13D+P\Big)^pv=3^{-p}D^pv+\sum_{j=0}^{p-1}\left(\frac13\right)^jD^jPv,
$$
and, since $v_0=1$,
$$
v_p=3^{-p}(D^p)_{p0}v_0+\sum_{j=0}^{p-1}\left(\frac13\right)^j(D^j)_{p0}v_0=\sum_{j=0}^{p}\left(\frac13\right)^j(D^j)_{p0}.
$$
Since, by definition, the entries of $D$ are either 0 or 1, we may interpret $D$ as the adjacency matrix of an oriented simple graph, where the vertices are the indices $0,1,\dots, |\s|$ and an oriented edge goes from $j$ to $i$ if $D_{ij}=1$. Then, $(D^j)_{p0}$ is equal to the number of oriented paths of length $j$  joining $0$ with $p$. Since from any vertex may depart at most three edges, if there is an edge joining $0$ with $p$, then there are at most 2 oriented paths of length 2 joining $0$ with $p$, at most 6 oriented paths of length 3 joining $0$ with $p$, and so on. So, denoting with $n_i$ the number of oriented paths of length $i$ joining $0$ with $p$, we have
\begin{equation}
\begin{cases}
n_1\leq1\\
n_1+n_2\leq3\\
3 n_1+n_2+n_3\leq9\\
\cdots\\
\sum_{i=1}^{q-1} 3^{q-1-i}n_i+n_q\leq 3^{q-1}.
\end{cases}
\end{equation}
As a consequence, for $q\geq 1$, we have
\begin{align*}
v_q=\sum_{i=1}^{q}3^{-i}n_i=
3^{-q}n_q+3^{1-q}\left(\sum_{i=1}^{q-1} 3^{q-1-i}n_i\right)
\leq3^{-q}n_q+3^{1-q}\left( 3^{q-1}-n_q\right)\leq1-\frac23 3^{1-q}n_q\leq1.
\end{align*}
\end{proof}

\subsection{$\forme$ embeds in the tangent module}\label{QgivesNorm}
We introduce here a completion of $\forme$ w.r.t. a given norm. This completion (and norm) will play only an auxiliary role, being used in the proof  that $\forme$ can be equivalently defined as the quotient of $\uforme$ w.r.t. the $\ch$-norm. But, this completion has some pathologies, cf. Proposition \ref{NoNorm}, in particular does not embed in $\ch$, therefore such norm will be abandoned later on.
By making use of the quadratic forms $Q_n$ defined in \eqref{nquadraticform}, we endow $\forme$ with  the norm
\begin{equation}\label{2infinitynorm}
\normaInf{ \o }=\sup_nQ_n[\o]^{1/2}.
\end{equation}
Since $Q_n\to Q$ on $\forme$, $\normaInf{ \o }$ is finite on it. Since $\normaInf{ \o }=0\Rightarrow Q_n[\o]=0,\ \forall n\Rightarrow \int_e\o=0,\ \forall e\in E_*$, the norm property follows.
Let us observe that the integrals $\o\to\int_e\o$ and the seminorm $\normaQ$ are continuous w.r.t. the norm $\normaInf{ \cdot }$.
We denote by $\ovforme$ the completion of $(\forme,\normaInf{ \cdot })$. Clearly, the quadratic forms $Q$, $Q_n$, $n\in \bn$, extend to $\ovforme$ by continuity.

\medskip

Let us now consider the space $\ell_N(\S):=\{a \in \br^\Sigma : N(a)<\infty\}$, where the functional $N$ is given by $N(a)=\displaystyle\sup_{n\ge 0} (5/3)^{n}\sum_{|\s|=n}|a_\s|$. Clearly, $N$ is a norm on such a space.
\begin{lem}\label{Nestimate}
Upper triangular matrices on $\S$ with bounded entries belong to $B(\ell_N(\S))$.
\end{lem}
\begin{proof}
Let $T$ be such a matrix, $v\in\ell_N(\S)$, so that $\sum_{|\t|=n}|(Tv)_\t|$
 $\leq$ $\sum_{|\t|=n}\sum_{\s\geq\t}|T_{\t\s}|\cdot|v_\s|
$ $\leq$ $\|T\|_\infty\sum_{|\s|\geq n}|v_\s|$
$\leq$ $ \|T\|_\infty N(v_\bullet)\sum_{k\geq n}(3/5)^k$
$=$ $\frac52\|T\|_\infty N(v_\bullet)(3/5)^n$, where $\|T\|_\infty=\sup_{\s\t}|T_{\s\t}|$,
implying that $N(Tv)\leq\frac52\|T\|_\infty N(v_\bullet)$.
\end{proof}
\begin{lem}\label{Lemma:csigmaestimate}
The sequence $\{c_\s:= \int_{\ell_\s} \o\}$  of periods of a smooth 1-form $\o$ belongs to $\ell_N(\S)$.
\end{lem}
\begin{proof} It is enough to prove the result for $\o=fdg$. Observe that
\begin{align*}
|c_\s|&=|\lim_n I_n(\ell_\s)(fdg)|
\leq |I_{|\s|+1}(\ell_\s)(fdg)|+\sum_{k=|\s|+1}^\infty| I_{k+1}(\ell_\s)(fdg)-I_k(\ell_\s)(fdg)|\,.
\end{align*}
Since $\ell_\sigma$ is a closed curve, $|I_{|\s|+1}(\ell_\s)(fdg)|=|I_{|\s|+1}(\ell_\s)((f-const)dg)|$. Denoting by $x_1,x_2,x_3$ the vertices of $\ell_\sigma$, and choosing $const=f(x_1)$, we get
$$
|I_{|\s|+1}(\ell_\s)(fdg)|
=|df(x_1,x_2) dg(x_1,x_2)+ df(x_1,x_3) dg(x_2,x_3)|
\leq \frac12 \sum_{e\in E_{|\s|+1}(\ell_\s)} df(e)^2+ dg(e)^2.$$
By \eqref{eq:b} we get
$\displaystyle |c_\s|\leq \frac12 \sum_{k=|\s|+1}^\infty\sum_{e\in E_{k}(\ell_\s)} \big( df(e)^2 + dg(e)^2 \big)$,
whence
$\sum_{|\s|=n}|c_\s|\leq\frac54\big(\frac35\big)^{n+1}\big( \ce[f] +\ce[g] \big)$. The thesis follows.
\end{proof}
\begin{lem}\label{prop:conditions}
If $c=\{c_\s\}$ belongs to $\ell_N(\S)$, then $k:=A^*c\in\ell_N(\S)$.
\end{lem}
\begin{proof}
Immediate  by Lemmas \ref{lem:asigmatau} and  \ref{Nestimate}.
\end{proof}

\begin{prop}\label{cor:decomp}
Let $c=\{c_\s\}$ belong to $\ell_N(\S)$, and set $k=A^*c$.
Then,  the series $\sum_\s k_\s dz_\s$ converges
to a  form $\o_H\in \ovforme$, having the $c_\s$'s as its periods.
In particular, if $\o\in\forme$, $c_\s:=\int_{\ell_\s}\o$,  $k$ and $\o_H$ as above, then $\o$ and $\o_H$ have the same periods.
\end{prop}
\begin{proof}
A simple calculation shows that $Q_n[dz_\s]\leq(5/3)^{|\s|}$, therefore
$\normaInf{k_\s dz_\s}^2$ $=$ $\sup_nQ_n[k_\s dz_\s]$ $\leq$ $|k_\s |^2(5/3)^{|\s|}$.
Then the series converges absolutely in $\ovforme$, since, by Lemma \ref{prop:conditions},
$$
\sum_\s\normaInf{k_\s dz_\s}\leq
\sum_\s (5/3)^{|\s|/2}|k_\s|\leq N(k_\bullet)\sum_k(3/5)^{k/2}
=\left(1-\sqrt{3/5}\right)^{-1}N(k_\bullet).
$$
In particular, $\int_{\ell_\t}\o_H=\sum_\s k_\s \int_{\ell_\t}dz_\s$.
By the results in Section \ref{sec:winding}, $AB=BA=1$, $|A_{\s\t}|\leq1$ and $|B_{\s\t}|=|\int_{\ell_\t}dz_\s|\leq1$, hence $A^*,B^*\in B(\ell_N(\S))$, by Lemma  \ref{Nestimate}. Then,
$$
\int_{\ell_\t}\o_H
=\sum_{\s}B_{\s\t} \sum_{\r\geq\s} A_{\r\s}  c_\r
=  (B^*A^*c)_{\t}
=((AB)^*c)_{\t}=c_{\t}
.$$
\end{proof}
\begin{lem}\label{prop:boundaries}
Let $\o$ be a smooth 1-form. Then, for any $\s$,
$$
\int_{\perim C_\s}\o= -\sum_{\t\geq\s}\int_{\ell_\t}\o=-\sum_{\t\geq\s}\int_{\ell_\t}\o_H=\int_{\perim C_\s}\o_H.
$$
\end{lem}
\begin{proof}
As above, we may assume $\o=fdg$. As for the first equation, we have, for any $n\geq|\s|$,
$$
\int_{\perim C_\s}fdg= -\sum_{\t\geq\s,|\t|\leq n}\int_{\ell_\t}fdg+\sum_{\t\geq\s,|\t|= n+1}\int_{\perim C_\t}fdg.
$$
Therefore we have to prove that the second summand goes to 0 when $n\to\infty$. It is not restrictive to assume $\s=\emptyset$. With estimates similar to those in Lemma \ref{Lemma:csigmaestimate}, we get
\begin{align*}
\sum_{|\t|= n+1}\int_{\perim C_\t}fdg
&\leq\frac12 \sum_{|\t|=n+1}\sum_{k=n+1}^\infty\sum_{e\in E_{k}(\perim C_\t)} df(e)^2+ dg(e)^2
\\
&\leq\frac12 \sum_{k=n+1}^\infty\left(\frac35\right)^{k}\left(\ce[f]+\ce[g]\right)
\leq \frac34 \left(\frac35\right)^{n}\left(\ce[f]+\ce[g]\right).
\end{align*}
The second equation  follows by  Proposition \ref{cor:decomp}, and the third by absolute convergence.
\end{proof}

\begin{lem}\label{lem:pathestimate}
Let $\o,k_\s,\o_H$ be  as above, and let $\g$ be an elementary simple path contained in the cell $C_\s$, $|\s|=n$. Then,
$$
\bigg| \int_\g\o_H \bigg| \leq N(k_\bullet)(n+3)\left(\frac35\right)^{n}.
$$
\end{lem}
\begin{proof}
It is easy to see that $\int_\g dz_\t$ can be non-zero only if either $\t<\s$ or $\t\geq\s$. Moreover, since $\g$ has no loops, $|\int_\g dz_\t|\leq1$.

When $\t<\s$, choosing $i$ such that $\t i\leq\s$, $dz_\t$ is exact in $C_{\t i}$, hence in $C_\s$, with $\osc_{C_{\t i}}(z_\t)=1/3$. Since the oscillation of a harmonic function in a sub-cell is bounded by $3/5$ times the oscillation of the original cell (cf. e.g. ex. 1.3.6 p. 8 in \cite{St06}), we get
\begin{equation}\label{ineq:osc}
\bigg| \int_\g dz_\t \bigg| \leq\osc_{C_{\s}}(z_\t)\leq \frac13 \Big( \frac35 \Big)^{|\s|-|\t|-1}.
\end{equation}
Since, by Lemma \ref{prop:conditions}, $\{k_\s\}\in\ell_N(\S)$,
\begin{align*}
\bigg| \int_\g \o_H \bigg|
&\leq
\sum_\t |k_\t|  \bigg| \int_\g dz_\t \bigg|
\leq \sum_{\t\geq\s} |k_\t| +
\frac13\sum_{\t<\s} |k_\t| \left(\frac35\right)^{n-|\t|-1}
\leq \sum_{|\t|\geq n} |k_\t| +
\frac13\sum_{|\t|< n} |k_\t| \left(\frac35\right)^{n-|\t|-1}\\
&\leq N(k_\bullet)\sum_{j\geq n} \left(\frac35\right)^{j} +
\frac13N(k_\bullet)\sum_{j< n} \left(\frac35\right)^{n-1}
\leq N(k_\bullet)(n+3)\left(\frac35\right)^{n}\,.
\end{align*}
\end{proof}

Let us now consider the form $\omega_1=\o-\o_H$, which has trivial integral along the perimeter of any cell $C_\s$. For any $n$, denoting by $S_n$  the 1-skeleton of the $n$-th approximation of $K$, given two points $x,y\in S_n$, and  a path $\g$ in $S_n$ joining them, the integral $\int_\g (\o-\o_H)$ depends only on the end points $x,y$, namely we get a primitive function $U_E^n$ on
 $S_n$, i.e,
 \begin{equation}\label{U1n}
 \forall e\in E_n,\qquad\int_e (\o-\o_H)= dU_E^n(e).
 \end{equation}

\begin{lem}\label{U1constr}
Let $\o=fdg$, $\o_H$ and $U_E^n$ be as above. Set $|\s|=n$, and choose $x_0\in V_n\cap C_\s$, $x\in V_{n+p}\cap C_\s$. Then there exists a constant $c$ such that
$$
|U_E^{n+p}(x)-U_E^{n+p}(x_0)|\leq \|f\|_\infty\osc_{C_\s}(g)+c (\ce[f]+\ce[g])(n+3) (3/5)^n.
$$
\end{lem}
\begin{proof}
{\bf First step}. Let $\s^0,\s^1,\dots\s^{p}$ be the subsequent multi-indices of length $n+j$, $\s^0=\s$, such that $x\in C_{\s^j}$, $j=0,\cdots,p$. We shall  construct inductively a path $\g$, joining $x_0$ with $x$, given by vertices $x_0,\dots x_{p+1}=x$, such that
\begin{itemize}
\item $x_j\in V_{n+j}$ for $j\leq p$, $x_{p+1}\in V_{n+p}$;
\item $x_j\in C_{\s^j}$, $j\leq p$;
\item either $x_{j-1}=x_{j}$, or
$x_{j-1},x_{j}$ are joined by an edge $e_{j}$, with $e_j\in E_{n+j}$ if $0<j\leq p$, and $e_{p+1}\in E_{n+p}$. In the first case we set $e_j$ to be the trivial edge.
\end{itemize}
Since $x_0$ is given, we only need to describe the inductive step. Suppose we have $x_{j-1}$, $j\leq p$. If $x_{j-1}\in  C_{\s^{j}}$, we set $x_{j}:=x_{j-1}$. If not, it is connected by an edge $e_{j}\in E_{n+j}$ to a vertex $x_j\in V_{n+j}\cap C_{\s^{j}}$. Finally, $x_{p}$ and $x_{p+1}$ are both vertices in $V_{n+p}\cap C_{\s^{p}}$, hence either coincide or are joined by an edge $e_{p+1}$.
\\
{\bf Second step}. There exists a constant $c_1$ such that
$$
|\int_\g fdg|\leq \|f\|_\infty\osc_{C_\s}(g)+c_1\left(\frac35\right)^n(\ce[f]+\ce[g]).
$$
We decompose the restriction of $f$ to $\g$ as $f=\sum_{k=0}^{p+1}f_k$, with $f_0=f(x_0)$ constantly, and, for $0<k\leq p+1$,
$$
f_k(t)=
\begin{cases}
0&t\in e_j,j<k,\\
f(t)-f(x_{k-1})&t\in e_k,\\
f(x_k)-f(x_{k-1})&t\in e_j,j>k.\\
\end{cases}
$$
We then get
\begin{align*}
\int_\g fdg
=&\int_\g f_0dg +
\sum_{k=1}^{p+1}\sum_{j=k}^{p+1}\int_{e_j} f_kdg\\
=& f(x_0)(g(x)-g(x_0)) +
\sum_{k=1}^{p+1}\int_{e_k} f_kdg
+\sum_{k=1}^{p}\sum_{j=k+1}^{p+1}  df(e_k) dg(e_j)\\
\end{align*}
As for the first summand, we clearly have $| f(x_0)(g(x)-g(x_0))|\leq \|f\|_\infty\osc_{C_\s}(g)$.
We now estimate the second summand.
First observe that
\begin{align*}
\bigg|  \int_{e_k} f_k dg \bigg| &
\leq   \left(I_{n+k}(e_k)(f_k dg)
+\sum_{r=n+k+1}^{\infty}|I_r(e_k)(f dg)-I_{r-1}(e_k)(f dg) |\right)\\
\leq   &\sum_{r=n+k}^{\infty}\left(\sum_{e\in E_r}  df(e)^2\right)^{1/2}
\left(\sum_{e\in E_r} dg(e)^2\right)^{1/2}
\leq  \frac{5}4\left(\frac35\right)^{n+k}(\ce[f]+\ce[g]).
\end{align*}
Therefore,
\begin{align*}
\bigg| \sum_{k=1}^{p+1} \int_{e_k} f_k dg \bigg|
\leq  \sum_{k=1}^{p+1} \frac{5}4\left(\frac35\right)^{n+k}(\ce[f]+\ce[g])
 \leq  \frac{15}8\left(\frac35\right)^{n}(\ce[f]+\ce[g]).
\end{align*}
We now consider the third summand. Since, $\forall e\in E_m$,
$| df(e)|\leq
(3/5)^{m/2} \ce[f]^{1/2}$, we get
\begin{align*}
|\sum_{k=1}^{p}\sum_{j=k+1}^{p+1}  df(e_k) dg(e_j)|
\leq
&\ce[f]^{1/2}\ce[g]^{1/2}\sum_{k=1}^{\infty} \left(\frac35\right)^{(n+k)/2}
\sum_{j=k+1}^{\infty} \left(\frac35\right)^{(n+j)/2}
\\
=
&\frac34\frac{\sqrt3}{\sqrt5-\sqrt3}\left(\frac35\right)^{n}(\ce[f]+\ce[g]).
\end{align*}
The thesis follows.
\\
{\bf Conclusion}.
Since
$|U_E^{n+p}(x)-U_E^{n+p}(x_0)| = | \int_\g (fdg-\o_H) | \leq
| \int_\g fdg |+ | \int_\g \o_H |$,
the result follows by Step 2 and Lemma \ref{lem:pathestimate}.
\end{proof}

\begin{prop}\label{cor:decompBis}
For any $\omega\in\forme$,  there exists $U_E\in\cf$ and $\o_H\in\forme$ such that
$\o=dU_E+\o_H$,
where $\o_H=\sum_\s k_\s dz_\s$.
\end{prop}

\begin{proof} As usual, it is not restrictive to assume $\o=fdg$. Clearly, the functions $U_E^n$ constructed above are defined up to an additive constant, therefore we choose a vertex $x$ in $V_0$ and set $U_E^n(x)=0$ for any $n$. Let us now observe that the functions $U_E^n$ satisfy, for $m\geq n$, $U_E^m|_{S_n}=U_E^n$, therefore they define a function $U_E$ on $S:=\cup_n S_n$. By Lemma \ref{U1constr}, $U_E$ is uniformly continuous on a dense subset of $K$, hence it extends to a continuous function on $K$, and, by definition, $\int_e (\o-\o_H)= dU_E(e)$.
This shows that
$$
Q_n[\o-dU_E-\sum_{|\s|\leq k}k_\s dz_\s]=Q_n[\sum_{|\s|> k}k_\s dz_\s],
$$
therefore, reasoning as in the proof of Proposition \ref{cor:decomp},
\begin{align*}
\normaInf{ \o-dU_E-\sum_{|\s|\leq k}k_\s dz_\s }
&\leq\sup_n \sum_{|\s|> k}Q_n[k_\s dz_\s]^{1/2}
=N(k_\bullet)\left(1-\sqrt{3/5}\right)^{-1}\left(\frac35\right)^{(k+1)/2}\to 0.
\end{align*}
This shows 
that $ \o-\o_H=d U_E$ as elements of $\ovforme$ and $\ce[U_E]=Q[\o-\o_H]<\infty$.
\end{proof}
\begin{thm}\label{cor:norm}
For $\o\in\forme$, $\int_e\o=0\ \forall e\in E_*$ {\it iff} $Q[\o]=0$, i.e.
$\normaQ$ is a norm in $\forme$.
Therefore, $\forme$ coincides with $\uforme/\{Q=0\}\subset\ch$, $Q^{1/2}$ coincides with $\|\cdot\|_\ch$, and $\ov{\forme}^\ch=\ch$.
\end{thm}
\begin{proof}
Since, by Proposition \ref{wclosed} $(i)$, the decomposition $\o=dU_E+\sum_\s k_\s dz_\s$ is an orthogonal decomposition w.r.t.~the norm $\normaQ$, then  $Q[\o]=0$ implies $Q[dU_E]=\ce[U_E]=0$ and $k_\s=0$ for any $\s\in\S$. As a consequence $\o$ vanishes. The last statement follows by the definition of $\ch$.
\end{proof}
Let us recall that $Q^{1/2}=\lim_n Q_n^{1/2}$ while $\|\cdot\|_{\rm sup}=\sup_n Q_n^{1/2}$. Since the second norm is stronger than the first, the first extends by continuity to a functional on the completion  $\ovforme$ of $\forme$ w.r.t. the norm $\|\cdot\|_{\rm sup}$. Contrary to the case of Sobolev spaces, where completions w.r.t. stronger norms imbeds into those with weaker norms, here Cauchy sequences which are equivalent  w.r.t. the weaker norm are not so w.r.t. the stronger, so that  $Q^{1/2}$ is only a seminorm on the completion  $\ovforme$ of $\forme$ w.r.t. the norm $\|\cdot\|_{\rm sup}$.
\begin{prop}\label{NoNorm}
$\normaQ$ is not a norm on $\ovforme$.
\end{prop}

\begin{proof}
We illustrate the strategy of the proof. First, we describe a sequence $\o_n$ in $\forme$, then we construct a normed space $A$ in which $(\forme, \normaInf\cdot)$ isometrically embeds, and a unit vector $\o\in A$ such that $\normaInf{\o-\o_n}\to0$, showing that the sequence $\o_n$ is Cauchy and $\lim_n\normaInf{\o_n}=1$. Finally, we observe that $\lim_n Q[\o_n]=0$.

Let $p_i$, $i=0,1,2$, be the external vertices of the gasket, $e_i$ be the edge in $E_0$ opposite to $p_i$, $i=0,1,2$, and let $g$ be the 0-harmonic function taking value $-1/2$ on $x_0$, 0 on $x_1$ and $1/2$ on $x_2$. Then, for any given $n$, let us consider the $n$-exact form $\o_n$ determined by the functions $f_\s$, $|\s|=n$, where
\begin{equation}
f_\s=
\begin{cases}
2^{-n}g\circ (w_\s)^{-1}&\text{if }\s\in\{0,2\}^n\\
0&\text{otherwise.}
\end{cases}
\end{equation}
Observe that, for any edge $e\in E_k$,
\begin{equation}
\lim_n\int_e\o_n=
\begin{cases}
2^{-k}&\text{if }e=w_\s e_1,\s\in\{0,2\}^k\\
0&\text{otherwise.}
\end{cases}
\end{equation}

\noindent Let us consider the vector space $\V(E)$ given by finite linear combinations of edges, its algebraic dual $\V(E)^*$, where the duality is denoted by the integral, $\langle\omega,e\rangle=\int_e\omega$, and the subspace $A=\{\omega\in \V(E)^*: \normaInf{ \o } <\infty\}$, with
$\normaInf{ \o }^2=\sup_n(5/3)^n\sum_{e\in E_n} | \int_e\o |^2$. Clearly $\normaInf{ \cdot }$ is a norm on $A$, $(A,\normaInf{ \cdot })$ is a normed vector space, and $\forme\subset A$ in an obvious way.
We now prove that $\o_n$ converges to a non-trivial element $\o\in A$, thus showing that  $\o_n$ is a Cauchy sequence in $\forme$ having a non-trivial limit $\o\in\ovforme$.

\noindent
Define $\o\in\V(E)^*$ by $\int_e \o := \begin{cases}
2^{-k}&\text{if }e=w_\s e_1,\s\in\{0,2\}^k\\
0&\text{otherwise.}
\end{cases}$. Since $Q_k[\o] = (5/6)^{k}$, $\omega$ is a unit element in $A$. We now  compute $Q_k[\o-\o_n]$ in $A$.

\noindent  If $k<n$,
$$
Q_k[\o-\o_n] = \Big( \frac53 \Big)^k \sum_{\substack{e\in E_k\\ e\not\subset e_1}} \Big| \int_e \o_n \Big|^2 = \Big( \frac53 \Big)^k \cdot 2^{k+1} (2^{-n-1})^2 = \frac12 \Big( \frac{10}3 \Big)^k 4^{-n}<\frac12 \Big( \frac56 \Big)^{n}.
$$
If $k\geq n$, we use the estimate $Q_k[\o-\o_n]^{1/2} \leq Q_k[\o]^{1/2} + Q_k[\o_n]^{1/2}$. Since each edge $e\in E_k$ is contained in only one cell $C_\s$, where $\o_n$ has a potential $f_\s$, we have
$$
Q_k[\o_n]
= \Big(\frac53 \Big)^k \sum_{|\s|=n} \sum_{e\in E_k(C_\s)} \Big| \int_e \o_n \Big|^2
= \sum_{\s\in\{0,2\}^n}\ce_{C_\s}[2^{-n}g\circ w_\s^{-1}]
=\frac32\left(\frac56\right)^n.
$$
In particular,  $Q_k[\o_n] = Q[\o_n]$.
Therefore, when $k\geq n$,
$$Q_k[\o-\o_n] \leq 2 Q_k[\o] + 2 Q_k[\o_n] \leq 2 \left(\frac5{6}\right)^{k} + 3 \left(\frac56 \right)^n = 5 \left(\frac56 \right)^n.$$
Hence,
$\normaInf{ \o-\o_n } ^2 = \sup_k Q_k[\o-\o_n] \leq 5 \left(\frac56 \right)^n$,
namely $\o_n$ converges in $\normaInf{ \cdot }$ to the non-trivial 1-form  $\o\in\ovforme$.
On the other hand, since $Q$ is continuous w.r.t. the norm $\|\cdot\|_{sup}$,
$Q[\o]=\lim_n Q[\o_n]=\lim_n\frac32\left(\frac56\right)^n=0$.
\end{proof}

\subsection{Hodge and De Rham Theorems}\label{HodgeDeRham}
Corollary \ref{cor:norm} shows that $\forme$ can be equivalently defined as the quotient of $\uforme$ w.r.t. the quadratic form $Q$, hence is a dense $\cf$-sub-module of $\ch$. Then $\forme$ may be considered as the space of smooth 1-forms, on which the integral along elementary paths is naturally defined. The following Lemma is the analytic counterpart of the fact that the gasket is topologically 1-dimensional.

\begin{lem}\label{DifferentialOnOmega1}
Any local exterior differential on $\forme$ which is a closable operator on $\ch$ vanishes, hence co-closed forms are harmonic.
\end{lem}
\begin{proof}
A closable operator $(d_1 ,\forme)$ on $\ch$, with values in another non degenerate, Hilbertian $\cf$-module $\O^2 (K)$, and giving rise to a complex $0\to \cf\to \forme\to \O^2 (K)$, necessarily vanishes on locally exact smooth 1-forms. However a form is locally exact {\it iff} the $k_\s$'s are eventually zero, as shown in Proposition \ref{wclosed} $(ii)$, therefore they are dense in $\ch$. The result follows.\end{proof}

\begin{thm}[Hodge decomposition] \label{cor:Hodge}
 Any  1-form $\o\in\ch$ can be uniquely decomposed as an orthogonal  sum $dU_E\oplus\o_H$ of an exact form and a harmonic form. The $dz_\s$'s give an orthogonal basis for the space of harmonic forms, therefore the decomposition above may be written as $\o=dU_E+\sum_\s k_\s dz_\s$.
\end{thm}
\begin{proof}
We observe that the space $B^1(K)$ of exact (smooth) forms is norm closed. This has been argued in \cite{CiSa2}, and we show it here for the sake of completeness. Indeed, $B^1 (K)$ is the range $d (\cf)$ of the derivation $d :\cf\to\ch$. Since the space of $0$-harmonic functions on $K$ is three dimensional, it is enough to prove that the image $d (\cf_0)$ of the subspace $\cf_0 :=\{f\in\cf : f\,\,{\rm vanishes\,\, on}\,\,V_0 \}$ of finite energy functions vanishing at the boundary $V_0$ of $K$, is closed in $\ch$. By the inequality
$
\|u\|_\infty\le c\sqrt{\ce [u]}\qquad u\in\cf_0
$
(holding for a finite constant $c>0$, see \cite{Kiga} Chapter 2), if $\{u_n\in\cf_0 :n\ge 1\}$ is a sequence such that $\{d u_n\in\ch :n\ge 1\}$ has the Cauchy property, then $\{u_n\in\cf_0 :n\ge 1\}$ is itself a Cauchy sequence in $\cf_0$  with respect to the uniform norm and we may consider its limit $u\in\cf_0$. As the quadratic form $\ce$ comes from a harmonic structure on $K$ (see \cite{Kiga} Example 3.1.5), it is the pointwise monotone limit of bounded quadratic forms on $C(K)$ and, in particular, it is lower semicontinuous. Then, if for a fixed $\varepsilon >0$, $N\geq 1$ is such that $\ce [u_n - u_m]<\varepsilon$, for all $n,m\geq N$, then
\[
\|d u -d u_m\|_\ch^2 =\ce [u -u_m]\le \liminf_n\ce [u_n -u_m]<\varepsilon\qquad m\ge N
\]
so that the sequence $\{d u_n\in\ch :n\ge 1\}$ converges to $d u\in\ch$.
\\
Finally, the space $B^1(K)^\perp$ consists of co-closed forms, which are also closed by Lemma \ref{DifferentialOnOmega1}. The result follows.
\end{proof}

\begin{rem}
(1) An equivalent way to formulate Hodge decomposition theorem is that each cohomology class has a (unique) harmonic representative.
\\
(2) Hodge decomposition allows us to define a gradient $d^*$ on forms:
$$d^*\o=d^*(dU_E+\o_H)=\Delta U_E.$$
Observe that the domain and the range of $d^*$ depend on the corresponding data for $\Delta$.
\\
(3) Even though the $dz_\s$'s are parametrized by lacunas, they are not the dual basis of the lacunas, considered as a basis for the homology vector space, as follows by eq. \eqref{asigmatau}.
\end{rem}

In order to formulate the first and second theorems by de Rham we need to introduce a stronger norm on $\ch$ such that  the integral on elementary paths still makes sense on the closure of $\forme$ w.r.t. such norm.
With $N$ as in Section \ref{QgivesNorm}, we set
$\|\o\|_N=\ce[U_E]^{1/2}+N(k_\bullet)$ on the  square integrable forms $\o=dU_E+\sum_\s k_\s dz_\s$ for which this expression is finite. We write $\ch_N:=\{\o\in\ch:\|\o\|_N<\infty\}$, and note that $\forme\subset\ch_N$.
\begin{lem}
If $\{k_\s\}\in\ell_N(\S)$, then $\sum_\s k_\s dz_\s\in\ch$.
\end{lem}
\begin{proof}
Eq. \eqref{dzsigmanorm} gives
\begin{align*}
\|\sum_\s k_\s dz_\s\|^2_\ch
&=\sum_\s|k_\s|^2\|dz_\s\|^2_\ch
\leq\frac56\sum_{n\geq0}\left(\frac53\right)^n\Big(\sum_{|\s|=n}|k_\s|\Big)^2
 \leq\frac{25}{12}N(k_\bullet)^2\, .
\end{align*}
\end{proof}
Let us consider the norm $N'$ on sequences, dual to the norm $N$: $N'(a)=\displaystyle\sum_{n\ge 0}(3/5)^{n}\sup_{|\s|=n}|a_\s|$.
We shall say that a path $\g\subset K$ has {\it finite effective length} $\l(\g)$ if
\begin{equation}\label{ef_length}
\l(\g):=N'(\int_\g dz_\bullet)<\infty.
\end{equation}
\begin{lem}\label{edgeshavefinitelength}
Edges have finite effective length. Indeed,
$\l(e)\leq(3/5)^{n-1}(3+2n)/6$ if $e\in E_n$.
\end{lem}
\begin{proof}
Let $e\in E_n$, and let $C_\s$, $|\s|=n$, be the cell having $e$ as a boundary edge. Then $\int_e d\dg_\t$ is non-zero only if either $\t\geq\s$ or $\t<\s$. It is easy to see that, for $\t\geq\s$, $\int_e d\dg_\t\leq1/3$, such value being attained e.g. for $\t=\s$. For $\t<\s$, one can estimate $\int_e d\dg_\t$ with the oscillation of $\dg_\t$ on $C_\s$, which in turn is estimated by $(3/5)^{|\s|-|\t|-1}\osc_{C_\r}(\dg_\t)$, where $C_\r$ is the cell of level $|\t|+1$ containing $C_\s$. Since $\osc_{C_\r}(\dg_\t)=1/3$, we have $\int_e d\dg_\t\leq1/3\cdot(3/5)^{n-|\t|-1}$. This value is not necessarily attained, but it does e.g. when $e$ is one of the boundary edges for $\ell_\t$,  for the index $\t$ immediately preceding $\s$. Therefore,
\begin{align*}
N'(\int_edz_\bullet)
&=\sum_{k\geq0}\left(\frac35\right)^k\sup_{|\s|=k}|\int_edz_\s|
\leq\sum_{k<n}\frac13\left(\frac35\right)^k\left(\frac35\right)^{n-k-1}
+\sum_{k\geq n}\frac13\left(\frac35\right)^k
=\frac{3+2n}{6}\left(\frac35\right)^{n-1}
\end{align*}
\end{proof}
\begin{thm}\label{12deRham}
The integral  extends by continuity to any $\o=dU_E+\sum_{\s} k_\s dz_\s\in\ch_N$ and any path $\g$ with finite effective length, as
\begin{equation}\label{genintegral1}
\int_\g \o=\int_\g dU_E+\sum_{\s} k_\s\int_\g dz_\s.
\end{equation}
In particular, $k_\t=\sum_{\s\geq\t}A_{\s\t} \int_{\ell_\s}\o$, hence, for $\o\in\ch_N$, the decomposition of  Theorem \ref{cor:Hodge}  is determined by the periods of $\o$. Therefore:
\\
$[$de Rham first  theorem$]$ If $\{c_\s\}\in\ell_N(\S)$,  $\exists\,\o_H$ harmonic in $\ch_N$ such that $\int_{\ell_\s}\o_H=c_\s$.
\\
$[$de Rham  second theorem$]$ If $\o\in\ch_N$ and  $\int_{\ell_\s}\o=0$ for all $\s$, then $\o$ is exact.
\end{thm}
\begin{proof}
Since $\sum_{\s} |k_\s|\cdot|\int_\g dz_\s|\leq N(k_\bullet)N'(\int_\g dz_\bullet)$ the series $\sum_{\s} k_\s\int_\g dz_\s$ converges. By Proposition \ref{cor:decompBis},   eq. \eqref{genintegral1} extends the integral of smooth forms on elementary paths.
The last statement follows as in the proof of Proposition \ref{cor:decomp}.
\\
De Rham first  theorem: if $k=A^*c$, then $\{k_\s\}\in\ell_N(\S)$, by Lemma \ref{Nestimate}, so we get the thesis by setting $\o_H=\sum_\s k_\s dz_\s$.
\\
De Rham second  theorem is immediate, by $k_\t=\sum_{\s\geq\t}A_{\s\t} \int_{\ell_\s}\o$.
\end{proof}

\subsection{On the existence of non-locally exact forms}
On a manifold, all closed forms are locally exact, namely the difference between closed and exact forms cannot be detected locally. Due to its exotic topology, this is no longer true  on the gasket, as we show below.
\begin{lem}\label{lem:aijk}
Let $f_i$ be the 0-harmonic function on the gasket taking value 1 on the vertex $p_i$ and 0 on the others, and consider the scalar products $a_{ijk}:=Q(df_i,f_jdf_k)$, $i,j,k=0,1,2$. Then
$$
a_{ijk}=
\begin{cases}
1&\text{if}\ i=j=k;\\
-\frac12&\text{if } i=j\ne k\text{ or } i\ne j= k;\\
\frac12&\text{if}\ i=k\ne j;\\
0&\text{if the indices are pairwise different.}
\end{cases}
$$
\end{lem}
\begin{proof}
The result directly follows from the definition of $Q$ and eq. \eqref{rem:cs}, together with the relation
$$
2Q(df_i,f_jdf_j) = Q(df_i,d(f^2_j))=\langle\Delta f_i,f^2_j \rangle=
\begin{cases}
2&\text{if}\ i=j;\\
-1&\text{if } i\ne  j;
\end{cases}
$$
where we recall that $\Delta f_i$ is the sum of twice the Dirac measure concentrated on the vertex $p_i$ minus the Dirac measures concentrated on the other vertices.
\end{proof}

\begin{lem}\label{kemptyset}
With the notation of the previous Lemma,
$$
Q(dz_\emptyset,f_0df_1)=\frac1{15}.
$$
\end{lem}
\begin{proof}
Since  $dz_\emptyset$ is invariant under $2\pi/3$ rotations, we have $Q(dz_\emptyset,f_0df_1)=Q(dz_\emptyset,f_idf_{i+1})$ for any $i=0,1,2$, hence
\begin{align*}
Q(dz_\emptyset,f_0df_1)
&=\frac13\sum_{i=0,1,2}Q(dz_\emptyset,f_idf_{i+1})
=\frac59\sum_{i,j=0,1,2}Q(dz_\emptyset\circ w_j,(f_idf_{i+1})\circ w_j)\\
&=\frac53\sum_{i=0,1,2}Q(dz_\emptyset\circ w_1,(f_idf_{i+1})\circ w_1),
\end{align*}
where, in the last equality, we used the fact that $\sum_{i=0,1,2}Q(dz_\emptyset\circ w_j,(f_idf_{i+1})\circ w_j)$ does not depend on $j$.
A simple computation shows that
$dz_\emptyset\circ w_1=dg$, with $g=\frac16(-f_0+f_2)$,
$f_0\circ w_1=\frac15(2f_0+f_2)$,
$f_1\circ w_{1}=\frac15(2+3f_1)$,
$f_2\circ w_{1}=\frac15(f_0+2f_2)$.
As a consequence,
\begin{align*}
Q(dz_\emptyset,f_0df_1)
&=\frac1{15}
Q(dg,2df_0+4df_2+3f_1df_0+6f_1df_2+6f_0df_1+3f_2df_1+ \\
&\qquad  +2f_0df_0+2f_2df_2+f_0df_2+4f_2df_0)\\
&=\frac1{15}Q(dg,2df_2+3f_1df_2-3f_2df_1+ 3f_2df_0)\\
&=\frac2{15}\langle\Delta f_2,g\rangle+\frac1{30}Q(d(-f_0+f_2),f_1df_2-f_2df_1+ f_2df_0)
\end{align*}
where in the second equality we used the invariance of the scalar product under the reflection of the gasket which fixes $p_1$. By  Lemma \ref{lem:aijk} the second summand vanishes, while $\langle\Delta f_2,g\rangle=1/2$, proving the thesis.
\end{proof}

\begin{prop}\label{Prop:NonLocEx}
The form $f_0df_1$ is not locally exact, indeed all the coefficients $k_\s$ of the decomposition of Theorem \ref{cor:Hodge} are non-zero.
\end{prop}
\begin{proof}
Set $\a(g,h)=Q(dz_\emptyset,gdh)$. Since $dz_\emptyset$ is harmonic,
$$\a(g,h)=Q(dz_\emptyset,gdh)=Q(dz_\emptyset,d(gh))-Q(dz_\emptyset,hdg)=-Q(dz_\emptyset,hdg)=-\a(h,g).$$
Restricting this bilinear form to 0-harmonic functions, we get a bilinear antisymmetric form on $\br^3$  such that $\a(g,const)=0$ for any $g$.  Moreover it is non-trivial since, by Lemma \ref{kemptyset}, $\a(f_0,f_1)=1/15$. As a consequence, $\a(g,h)=0$ {\it iff} $ag+bh=1$, for some constants $a,b$. For any index $\s$ we get
$$
Q(dz_\s,f_0df_1)=\left(\frac53\right)^{|\s|}Q(dz_\emptyset,f_0\circ w_\s d(f_1\circ w_\s))
=\left(\frac53\right)^{|\s|}\a(f_0\circ w_\s ,f_1\circ w_\s).
$$
By harmonicity of $f_i$, the map $f_i\to f_i\circ w_\s$ is injective and linear, therefore
$\a(f_0 ,f_1)\neq0$ $\Leftrightarrow $ $f_0$ and $f_1$ do not generate constants $\Leftrightarrow $ $f_0\circ w_\s$ and $f_1\circ w_\s$ do not generate constants $\Leftrightarrow $ $\a(f_0\circ w_\s ,f_1\circ w_\s)\neq0$.
Finally, by Theorem \ref{cor:Hodge}, we have
$$
Q(dz_\s,f_0df_1)=k_\s Q[dz_\s],
$$
namely $k_\s\neq0$ for any $\s$. 
\end{proof}

\section{Potentials of smooth 1-forms}\label{Potentials}
The first aim of this section is to prove a de Rham duality Theorem for locally exact forms, namely when only finite linear combinations of exact forms and $dz_\s$'s are considered. In this ``algebraic'' case, the integral is defined for any path in $K$, and locally exact forms are in one to one correspondence with suitable affine potentials (up to additive constants) on a suitable pro-covering space.

\subsection{Uniform coverings of the Sierpinski gasket}\label{extendedIntegral}
Berestovskii and Plaut introduced a notion of Uniform Universal Cover for a suitable family of uniform spaces (such spaces are called coverable).
As explained in \cite{PlWi}, in the case of a connected metric space $X$, the inverse system giving rise to the Uniform Universal Cover $\widetilde X$ consists in a tower $\{X_\eps\}_{\eps>0}$ of regular coverings, where $\eps$ corresponds to the size of the cycles that are unfolded in the corresponding covering. When the space is also geodesic, the equivalence class of the covering actually changes only for a discrete set in $(0,\infty)$. For the gasket $K$ of side 1 endowed with the geodesic metric induced by the embedding in $\br^2$, the sequence is $\{\eps_n=const\cdot 2^{-n}\}$,   as can be easily deduced from the content of Section 7 of \cite{BerPla}. The projective limit of the groups $\deck(\KK_n)$, which is denoted by $\delta_1(K)$, and called the deck group of $K$ in \cite{BerPla}, coincides with the  {\it \v{C}ech homotopy group} $\check{\pi}_1(K)$ of $K$, cf. \cite{ADTW}, Proposition 2.8.

We now describe the coverings  $\widetilde{K}_n=K_{\eps_n}$, $n\in\bn$, of the gasket $K$.
Let us recall that, for any $n$,  $K$ can be written as
$
K=\bigcup_{|\sigma|=n} w_\sigma(K).
$
If $T$ is the  convex hull of $K$ in the plane, and $T_n=\bigcup_{|\s|=n}w_\s(T)$, $\KK_n$ may be seen as the regular  covering of $K$ induced by the universal covering $\wt{T}_n$ of $T_n$  via the embedding $\iota_n : K \hookrightarrow T_n$.
Due to the simple connectedness of $\widetilde T_n$, the local potentials $f_\s$, $|\s|=n$, of a (not necessarily smooth) $n$-exact form $\o$ glue together to form a continuous potential  $f_\o$ of $\tilde\o$ on $\KK_n$.
If $\widetilde\o$ is the $\deck(\KK_n)$-periodic form obtained by lifting $\o$ to $\KK_n$, we clearly have
\begin{equation}\label{eq:n-primitive}
\int_{\g}\o=\int_{\tilde\g}\tilde\o=f_\o(\tilde\g(1))-f_\o(\tilde\g(0)).
\end{equation}

\begin{dfn} (Affine functions)
Let $G$ be a topological group acting on a space $X$. A continuous function $f$ on $X$ is {\it $G$-affine} if there exists a continuous group homomorphism
$\f:G\to(\br,+)$ such that $f(gx)=f(x)+\f(g)$ for all $(g,x)\in G\times X$.
\end{dfn}

 Let us observe that, since the group homomorphisms $\f$ associated to affine functions are valued in the {\it abelian} group $(\br,+)$, they vanish on commutators. In particular, let $[\deck(\KK_n),\deck(\KK_n)]$ and $\G_n:=\ab(\deck(\KK_n))=\deck(\KK_n)/[\deck(\KK_n),\deck(\KK_n)]$ be the commutator subgroup and   the abelianization, respectively, of the group $\deck(\KK_n)$. Then, a $\deck(\KK_n)$-affine function on $\KK_n$ can be considered as a $\G_n$-affine function on the quotient space
\[
\LL_n :=\KK_n/[\deck(\KK_n),\deck(\KK_n)]\, ,
\]
which is an abelian covering $(\LL_n,p_n,K)$ of $K$  (cf. e.g. \cite{Sieradski}, p. 423, or \cite{Szamuely}, Theorem 2.2.10) such that
$\deck(\LL_n) = \G_n$.
The latter  is a free abelian group with as many generators as the number of lacunas $\ell_\s$, $|\s|\leq n-1$. Let us  mention that the abelian coverings $\LL_n$,  as well as their non-abelian counterparts $\KK_n$, are fractafolds in the sense of  Strichartz  \cite{Stri}  (see also  \cite{St09,StriTepla}).
 See figure \ref{FigFund} for a portion of $\LL_2$, which is an example of a fundamental domain in the sense of Proposition \ref{lem:sametop}. Notice that $x_i$ and $x_i'$ project to the same point on $K$.

     \begin{figure}[ht]
 	 \centering
	 \psfig{file=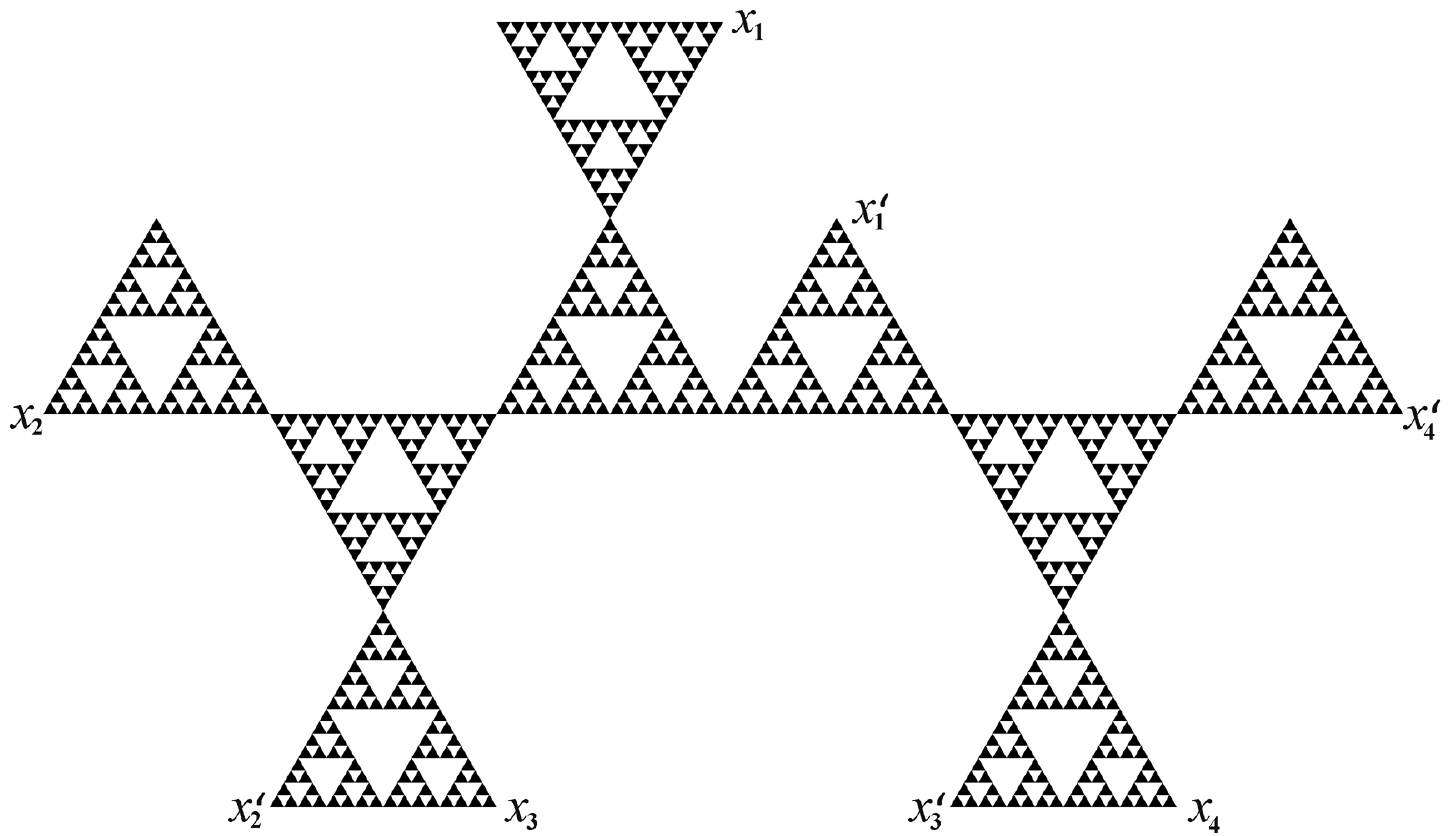,height=2.4in}
	 \caption{A fundamental domain for $\LL_2$.}
	 \label{FigFund}
     \end{figure}

\begin{lem}
The above constructed potential $f_\o$ of an $n$-exact topological 1-form $\o$ is a $\deck(\KK_n)$-affine function on the covering space $\KK_n$, hence it can be considered as a $\Gamma_n$-affine function on the abelian covering space  $\LL_n$.
\end{lem}
\begin{proof}
As already observed, the r.h.s. in (\ref{eq:n-primitive}) is clearly $\deck(\KK_n)$-invariant, namely
$$
f(x)-f(x_0)=f(gx)-f(gx_0),\quad \forall g\in\deck(\KK_n),
$$
or, equivalently, $f(gx_0)-f(x_0)=f(gx)-f(x)$, namely the quantity $\f(g)=f(gx)-f(x)$ only depends on the group element $g$, and gives rise to a function on the group $\deck(\KK_n)$, which is automatically continuous as this group is discrete. Moreover, for $g,h\in \deck(\KK_n)$,
$$
\f(gh)=f(ghx)-f(x)=\left(f(ghx)-f(hx)\right)+\left(f(hx)-f(x)\right)=\f(g)+\f(h),
$$
that is $\f$ is a homomorphism from $\deck(\KK_n)$ to $(\br,+)$.
\end{proof}

The family $\big\{(\LL_n,p_n,K):n\in\bn\big\}$ is  projective too, and we denote by $\widetilde L$ the projective limit space.
The projective limit $\G$ of the groups $\G_n$ is the direct product of countably many copies of $\bz$, where generators can be identified with lacunas, and coincides with the first {\it \v{C}ech homology group} (cf. e.g. \cite{EiSt}, Theorem X.3.1, p. 261), which we shall denote  by $\check{H}_1(K)$.
The group  $\check{\pi}_1(K)$ projects surjectively on $\check{H}_1(K)$.

\begin{dfn}\label{Gamma-affine}
We call {\it  \UUAC} of $K$ the projective limit $\LL=\displaystyle\lim_{\leftarrow}\LL_n$, topologized by the projective limit topology.
\end{dfn}

We list below some properties of the spaces $\KK$ and $\LL$ that are needed in the sequel. We refer to \cite{BerPla} for  other interesting properties.

\begin{prop}
\label{Properties}
\item{$(i)$} $\KK$ and $\LL$ have  the unique path-lifting property.
\item{$(ii)$} $\KK$ and $\LL$ are path-wise connected.
\item{$(iii)$} $\G$ is the direct product of countably many copies of $\bz$.
\end{prop}
\begin{proof}
Property $(i)$ for $\KK$ follows by Corollary 74 in \cite{BerPla}.
Property $(ii)$ for $\KK$ follows by Corollary 83 in \cite{BerPla}: indeed, since $K$ is geodesic, it is uniformly (locally) path-wise connected (cf. Definition 66 in \cite{BerPla}).
The corresponding properties of $\LL$ follow, since $\KK$ projects surjectively on $\LL$.
Property $(iii)$ follows by the definition of $\G$ and \cite{BerPla}, Section 7.
\end{proof}

\begin{lem}\label{finitelyAffine}
For any $\G$-affine function $f$ on $\LL$ there exists $n\in\bn$ and a $\G_n$-affine function $f_n$ on $\LL_n$ such that $f_n$ lifts to $f$.
\end{lem}
\begin{proof}
This is the same as saying that the homomorphism $\f$ associated with $f$ satisfies $\f(g_\s)=0$, for $|\s|$ large enough, where $g_\s$ denotes the homotopy class of the lacuna $\ell_\sigma$. Assume, by contradiction, that  $\f$ is continuous, and  non-trivial on infinitely many elements $g_n=g_{\s_n}$. Recall that a sequence $h_n$ in $\Gamma$ converges to $h$  in the projective limit topology iff, for any $k\in\bn$,  $q_k(h_n)=q_k(h)$ for sufficiently large $n$, where $q_k:\Gamma\to\Gamma_k$ is the projection; therefore, for any sequence $\set{k_n} \subset\bz$, $\lim_N\prod_{n=1}^N g_n^{k_n}=\prod_{n=1}^\infty g_n^{k_n}$ in the projective limit topology. As a consequence,
$$
\f\left(\prod_{n=1}^\infty g_n^{k_n}\right)=\sum_{n=1}^\infty k_n\f(g_n).
$$
However, one may always find a sequence of integers $\{k_n\}_{n\in\bn}$ such that the series above diverges.
\end{proof}

By the general theory of Dirichlet forms (see for example \cite{FOT}), the space of {\it locally finite energy functions} ${\widetilde\cf}_{n,{\rm loc}}$   on $\KK_n$ is  defined as those functions which coincide, on any open set of a suitable open cover of $\KK_n$, with a finite energy function in ${\widetilde\cf}_n $. Locally finite energy functions on $\KK_n$ are, in particular, continuous. Potentials of locally exact smooth forms on $K$ will be locally finite energy functions on the above considered covers.

\begin{lem}
$(i)$ A quadratic (energy) form $ \ce_\G : \ca (\Gamma, \LL)\to [0,+\infty]$ is well defined on the space $\ca (\Gamma, \LL)$ of $\G$-affine functions on the covering space $\LL$ by
\begin{equation}
\ce_\G[f]=\lim_n\left(\frac53\right)^n\sum_{e\in E_n}|\partial f(e)|^2\, ,
\end{equation}
where the quantity $\partial f(e):=f(\tilde e_+)-f(\tilde e_-)$ does not depend on the choice of the lifting $\tilde e\subset \LL$ of $e\in E_*(K)$.

\noindent $(ii)$ The energy of a $\Gamma$-affine function $f$ is finite if and only if  $f$ is the potential of a locally exact form $\o$ on $K$, and, in that case, $\ce_\G[f]=\|\o\|_\ch^2$. We shall write $df=\o$.
\end{lem}
\begin{proof} $(i)$ Let $\tilde e^1, \tilde e^2$ be two liftings, $\tilde e^1_n, \tilde e^2_n$ the corresponding projections on $\LL_n$, $g_n\in \G_n$ be such that $g_n (\tilde e^1_n)= \tilde e^2_n$. The family $\set{g_n}$ is a projective sequence of deck transformations, which defines a deck transformation $g$ on $\LL$ satisfying $g(\tilde e^1)= \tilde e^2$. Since $f$ is $\Gamma$-affine its variation is the same for all liftings. Since $f$ is the lifting of a continuous function on $\LL_m$ for some $m$, the sequence above is increasing for $n>m$, and this shows the second statement.

\noindent $(ii)$
If $f$ is a $\Gamma$-affine function of finite energy then, by Lemma \ref{finitelyAffine}, $f$ is the lifting of a $\G_n$-affine function $f_n$ on $\LL_n$. Set $f_\s={f_n}|_{C_\s}$ for $|\s|=n$. Since the covering projection from $\LL_n$ to $K$ is one to one on cells of level $n$, we get the desired form by glueing the $df_\s$'s. Conversely, the existence of a potential of a locally exact form has been already shown above, and the equality $\ce_\G [f] = \|\o\|^2_2$ follows  by Lemma \ref{Q-norm}.
\end{proof}

\noindent
Notice that the quadratic form just defined on $\Gamma$-affine functions on the covering space $\LL$ reduces to the standard Dirichlet form on the gasket $K$ when evaluated on periodic functions, i.e. on (liftings of) functions on $K$. This is also the reason why the notation $df=\o$ is consistent with the usual notation for the derivation of a finite energy function on $K$.

By Proposition \ref{wclosed} $(ii)$, any locally exact topological form modulo exact topological forms may be uniquely written as a finite linear combination of the $dz_\s$, the same result holding for locally exact smooth forms modulo exact smooth forms. Therefore, denoting by $B^1C(K)$ the space of exact topological 1-forms, the following definition makes sense.
\begin{dfn}
We define $B^1(K,\br)$ as the space of exact forms on $K$, and
\begin{equation}\label{samegroup}
H^1_{dR}(K,\br)=\frac{\toplex}{B^1C(K)}=\frac{\qex}{B^1(K)}
\end{equation}
as the {\it algebraic de Rham cohomology group} for the Sierpinski gasket.
\end{dfn}

\begin{rem}
Since the group $\G=\check{H}_1(K)$ has no torsion, its homological information is fully recovered by the group $\check{H}_1(K,\br)=\G\otimes_\bz\br$.
\end{rem}

\begin{thm}[de Rham cohomology theorem]\label{deRham}
There is a one to one correspondence between locally exact topological forms  and their potentials (up to additive constants) such that $\int_\g \omega=f_\o(x_1)-f_\o(x_0)$ for any path $\g\subset K$, where $x_0,x_1$ are the end-points of a lifting of $\g$ to $\LL$. When the locally exact form is smooth the corresponding potential has finite energy. Any class in $H^1_{dR}(K,\br)$ has a smooth representative.
The pairing $\langle\g,\o\rangle=\int_\g\o$ between continuous paths and locally exact forms  gives rise to a nondegenerate pairing between elements of the group $\check{H}_1(K,\br)$  and elements of  $H^1_{dR}(K,\br)$. Such a pairing is indeed a duality.
\end{thm}
\begin{proof}
The first and second statements follow by the Lemmas above. The third follows by eq. \eqref{samegroup}. As for the last statement, observe that, for any  continuous closed path $\g$ in $K$ and for any $n$, we may associate with $\g$ its singular homology class $[\g]_n\in H_1(T_n)$, and then the projective limit $[\g]=\displaystyle\lim_{\leftarrow}[\g]_n\in\G=\lim_{\leftarrow}H_1(T_n)$. If $\o$ is $k$-exact, and $\f_\o$ the associated homomorphism, then
$$
\f_\o([\g])=\langle\lim_{\leftarrow}[\g]_n,\o\rangle=\langle[\g]_k,\o\rangle=\int_\g\o.
$$
Since the pairing above is trivial when the form is exact, we get a pairing $\G\times H^1_{dR}(K,\br)\to\br$. Such pairing clearly extends to a pairing $\check{H}_1(K,\br)\times H^1_{dR}(K,\br)\to\br$.
\\
Now we prove the duality relation. On the one hand,
$H^1_{dR}(K,\br)$ is isomorphic to $\displaystyle\lim_{\rightarrow}H^1_{dR}(T_n,\br)$, topologized with the direct limit topology. On the other hand $\displaystyle\check{H}_1(K,\br)=\lim_{\leftarrow}H_1(T_n,\br)$, topologized with the projective limit topology. The thesis follows by the classical duality result for $T_n$.
\end{proof}

\subsection{A metric on the  \UUAC  $\widetilde L$}
In Section \ref{HodgeDeRham} we have seen that the introduction of the norm $N$ on sequences selects both the space $\ch_N$ of 1-forms and the class of paths with finite effective length in such a way that the corresponding integral exists and is finite. The notion of path with finite effective length may also be read on the  \UUAC  $\widetilde L$, where the norm on sequences induces a (possibly infinite) distance $d_N$, hence splits the space in $d_N$-components. A path has finite effective length {\it iff} its lifting to $\LL$ is contained in a single $d_N$-component. Forms with finite $\|\cdot\|_N$ norm have a finite, continuous potential on any $d_N$-component of $\LL$.

Clearly, by replacing the norm $N$ with another norm on sequences, we may enlarge the class of 1-forms which may be lifted to (exact) 1-forms on (any $d_N$-component of) $\LL$, the key property for which the construction works being the connectedness of $K$ by paths with finite  effective length. This property is not satisfied in an extreme way when $\|\cdot\|_N$ coincides with the  norm on $\ch$. Indeed, this choice will restrict the $d_N$-components in such a way that their projection to $K$ does not contain any edge, that is to say the potentials of such forms are defined in an extremely small space. Equivalently, no edge has finite effective length, cf.~Remark \ref{lastrem}.

We make use here of the norms $N$ and $N'$ on sequences $\{a \in \br^\Sigma\}$ introduced in Section \ref{HodgeDeRham}.
The metric $d_N$ considered in the following will take also the value $+\infty$, therefore it splits the space in {\it $d_N$-components}, namely maximal subsets of points  with mutually finite distance.
Denoting by $z_\s$ the $\G_n$-affine potential on $\LL_n$ of the $n$-exact form $dz_\s$, $n=|\s|+1$, and by $\f_\s :\Gamma\to\mathbb{R}$ the corresponding homomorphism, we consider the function
\begin{equation}\label{eq:qmetric}
d_N(x,y)= N'\big(z_\bullet(x)-z_\bullet(y)\big)\, .
\end{equation}

\begin{lem}
The function $d_N$ is a $\Gamma$-invariant metric which is finer than the projective limit topology. If $\g$ is a path in $K$ and $\tilde\g$ is a lifting on $\LL$, the effective length of $\g$ may be equivalently defined as $\l(\g):=d_N(\tilde\g(1),\tilde\g(0))$.
\end{lem}
\begin{proof}
The value $d_N(x,y)$ is obtained by composing the norm  $N'$ on sequences indexed by $\Sigma$ with the (semi-definite) distances $d_\s(x,y)=|z_\s(y)-z_\s(x)|$. Therefore, on the one hand $d_N$ is a (possibly semi-definite) metric on $\LL$, on the other hand  the topology induced by $d_N$ is stronger than the weak topology induced by the $z_\s$'s, which is the projective limit topology, by Lemma \ref{lem:sametop} in the Appendix. Since the projective limit topology is Hausdorff, this shows at once that $d_N$ is positive definite and that is  finer than the projective limit topology.
Finally, we have
$d_N(gx,gy)= N'\big(z_\bullet(gy)-z_\bullet(gx) \big)=
N'\big(z_\bullet(y)-z_\bullet(x)\big)=d_N(x,y)$ for all $g\in\Gamma$.
The last statement follows by the given definitions.
\end{proof}

\begin{lem}
Let $x$ be a point in $\LL$, $g\in \G$. Then, the quantity $\ell_N(g):=d_N(x,g x)$ does not depend on $x$, and $\ell_N(g)=0$ {\it iff} $g$ is the identity. The set $\G_N=\{g\in \Gamma:d_N(x,g x)<\infty\}$ does not depend on $x$, and is a subgroup of $\G$. The function $\ell_N(g)$ is a  length function on $\G_N$.
\end{lem}
\begin{proof}
For any $\s\in\Sigma$, let $\f_\s\in {\rm hom} (\Gamma ,\mathbb{R})$ be the homomorphism associated to the $\Gamma$-affine function $z_\s$ on $\widetilde L$ in such a way that $z_\s (gx)-z_\s (x)=\f_\s (g)$ for all $g\in\Gamma$. Let us denote by $\f_\bullet (g)\in \mathbb{R}^\Sigma$ the sequence $\s\mapsto\f_\s (g)$. Equation \eqref{eq:qmetric} then shows that
$$
d_N(x,g x)=N\big(z_\bullet(gx)-z_\bullet(x)\big)=N\big(\f_\bullet(g)\big),
$$
and the first statement follows. Since $d_N$ is a pseudo-metric, $\ell_N(g)=0$ means $gx=x$ for any $x$, namely $g=e$. The last two properties are obvious.
\end{proof}

\begin{lem} \label{ConnELungh}
	The projection map $p$ restricted to a $d_N$-component is surjective $\iff$ for all $x,y\in K$ there is a continuous path $\g$ in $K$ between them which has finite effective length.
\end{lem}
\begin{proof}
	$(\Longleftarrow)$ Let us fix $\wt{x}_0\in \LL$, and let $x_0:= p(\wt{x}_0)$. Then, for any $x\in K$ there is a continuous path $\g$ in $K$, starting in $x_0$ and ending in $x$,  which has finite effective length. Denote by $\wt{\g}$ its unique lifting to a path in $\LL$ starting at  $\wt{x}_0\in\LL$. Then $\pi(\wt{\g}(1))=x$, and $\wt{\g}(1)$ belongs to the same $d_N$-component of $\wt{x}_0$.
	
	\noindent $(\Longrightarrow)$ Let $x,y\in K$. By assumption, there are $\wt{x},\wt{y}\in\LL$  such that $d_N(\wt{x},\wt{y})<\infty$ and $p(\wt{x})=x$, $p(\wt{y})=y$. Because of Proposition \ref{Properties} $(ii)$, there is a continuous path $\wt{\g}$ in $\LL$ between $\wt{x}$ and $\wt{y}$. Set $\g:= p\circ \wt{\g}$, which automatically has finite effective length.
\end{proof}

\begin{lem}
Elementary paths have finite effective length. Any $d_N$-component of $\LL$ projects surjectively on $K$.
\end{lem}
\begin{proof}
By Lemma \ref{edgeshavefinitelength}, edges have finite effective length.
Now we observe that the effective length is sub-additive. Indeed, if $\g_1,\g_2$ are consecutive paths, $\wt{\g}_1$ is a lifting of $\g_1$ starting from some point $\wt{x}_0\in\LL$, and $\wt{\g}_2$ is a lifting of $\g_2$ starting from   $x:=\wt{\g}_1(1)$, then
$$
\l(\g_1\cdot\g_2)=d_N(\wt{\g}_1(0),\wt{\g}_2(1))\leq d_N(\wt{\g}_1(0),x)+d_N(x,\wt{\g}_2(1)) = \l(\g_1)+\l(\g_2).
$$
The first statement follows.
As for the second, the thesis is equivalent to the connectedness of $K$ by means of paths of finite effective length, as shown in Lemma \ref{ConnELungh}. We have shown in Lemma \ref{U1constr} that a vertex $v_0\in V_0$ can be connected to any vertex of level $p$ by an elementary path consisting of at most 1 edge of level $j$ for any $j\leq p$, thus proving that $v_0$ can indeed  be connected to any point $x$ in $K$ by a path consisting of (possibly infinitely many) edges, at most 1 of them for any level. The thesis follows by the estimate in Lemma \ref{edgeshavefinitelength} and sub-additivity.
\end{proof}

\subsection{Potentials of smooth 1-forms}
Indeed the results are formulated for the elements of the closure of $\forme$ in $\ch$ w.r.t. the norm $\|\cdot\|_N$, namely for elements of $\ch_N$.

\begin{lem}
Let $\o=dU_E+\sum_\s k_\s dz_\s\in\ch_N$.
For any $d_N$-component  $\LL_0\subset\LL$, we may  associate to $\o$ a function $U=U_E+U_H$ , where $U_E$ was described in Proposition \ref{cor:decompBis} and,  $\forall x_0\in\LL_0$, $U_H$ may be written as
$$
U_H(x)=\sum_\s k_\s  (z_\s(x)-z_\s(x_0)).
$$
The series defining $U_H$ converges uniformly on compact sets, and
 $U_H$ is a $d_N$-continuous $\Gamma_N$-affine function on $\LL_0$.
In particular,  $U$   is a potential for $\omega$, namely, for any continuous path $\g$ in $K$, $\l(\g)<\infty$, and any lifting $\wt\g$ of $\g$ to $\LL$, it holds
$\int_\gamma \omega=U(\tilde\g(1))-U(\tilde\g(0))<\infty$.
\end{lem}
\begin{proof}
Given two points $x_1,x_2\in\LL_0$, we have
\begin{align*}
|U_H(x_2)-U_H(x_1)|
&=\bigg| \sum_\s k_\s(z_\s(x_2)-z_\s(x_1)) \bigg|
\leq  N'(z_\bullet(x_2)-z_\bullet(x_1))\,N(k_\bullet)
\leq d_N(x_1,x_2)\|\o\|_N.
\end{align*}
As a consequence, $U_H$ is Lipschitz $d_N$-continuous. In particular, if $\ell(g)<\infty$, then $x$ and $gx$ belong to the same $d_N$-component, and $U_H(gx)-U_H(x)=\sum_\s k_\s\f_\s(g)$, namely $U_H$ is $\G_N$-affine.
Since $U_E$ is continuous on $K$, it lifts to a $\G$-invariant function on $\LL$, continuous in the projective limit topology, hence also in the (stronger) $d_N$-topology. The last statement easily follows.
\end{proof}

\begin{thm}\label{primitiveoncomponent}
$(i)$  Any form in $\ch_N$ has a $\G_N$-affine potential  on any $d_N$-component of  $\LL$;
\item{$(ii)$}  the integral of a form in $\ch_N$   along a path $\g$  with finite effective length coincides with the variation of the potential at the end points of a lifting of $\g$;
\item{$(iii)$}  such integral gives  a nondegenerate pairing between $\G_N$ and $\ch_N/B^1(K)$. Indeed, the space $\ch_N/B^1(K)$ is the Banach space dual of  $\G_N\otimes_\bz\br$.
\end{thm}
\begin{proof}
The first two statements have been proved above. As for the third, we observe that, for $g\in\G_N$, $\o\in\ch_N$, the pairing $\langle g,\o\rangle$ may be defined as $\int_\g\o$, where $\g$ is any closed path giving a representative of $g$. Then, $g\in\G_N$ {\it iff} the sequence $\{\langle g,dz_\s\rangle\}_{\s\in\S} $ belongs to $\ell_{N'}(\S)$.
On the other hand, we proved in Theorem \ref{deRham} that $\check{H}_1(K,\br)=\G\otimes_\bz\br$ may be identified with the dual of the space of cohomology classes of locally exact 1-forms.
Since any such class may be uniquely described through its harmonic representative $\o=\sum_\s k_\s dz_\s$, namely by the eventually zero sequence $k:=\{k_\s\}$,  the elements of $\G\otimes_\bz\br$ may be identified with infinite sequences $\a=\{\a_\s\}$, with $\langle\a,\o\rangle=\sum_\s\a_\s k_\s$.
Then, the sequence $\a=\{\a_\s\}$ belongs to $\G_N\otimes_\bz\br$ {\it iff} $\{ \a_\s\}_{\s\in\S}=\{\langle \a,dz_\s\rangle\}_{\s\in\S} \in\ell_{N'}(\S)$, or, equivalently, $\G_N\otimes_\bz\br \cong \ell_{N'}(\S)$. Since $\ch_N/B^1(K)$ may be identified with $\ell_{N}(\S)$, the thesis follows.
\end{proof}

We observe that, for any choice of the norm $N$ such that the $d_N$-components of $\LL$ projects surjectively on $K$, it is possible to construct potentials of any form in $\ch_N$ on the $d_N$-components. The choice of the norm $N$ affects both the size of the $d_N$-components and the smoothness of the forms that can be lifted there. In particular, we may lift less regular 1-forms but on smaller $d_N$-components. However,  lifting all forms in the tangent bimodule $\ch$ means getting $d_N$-components as small as to contain no edge.
\begin{prop}\label{lastrem}
Choosing $N$ such that $\|\cdot\|_N=\|\cdot\|_\ch$, no edge has finite effective length.
\end{prop}
\begin{proof}
We prove the statement for the edge $e_1$ opposite to the vertex $p_1\in V_0$, the other cases follow in a similar way.
The norms $\|\cdot\|_N$ and $\|\cdot\|_\ch$ coincide if we set $N(a)^2=5/6\sum_\s(5/3)^{|\s|}|a_\s|^2$.
Therefore $ \l(e_1)^2=6/5\sum_\s(3/5)^{|\s|}|\int_{e_1}dz_\s|^2$.
Since $\int_{e_1}dz_\s=-1/3$ if the multi-index $\s$ does not contain the index $1$, and vanishes otherwise, we get
$\l(e_1)^2=(2/15)\sum_n(6/5)^{n}=+\infty$.
\end{proof}

\appendix

\section{The projective limit topology on $\LL$ is generated by potentials}

\begin{lem}\label{lem:axes}
Let $C_{\s i}$ be one of the three subcells of the cell $C_\s$,  denote by $z_\s^i$ the potential of $dz_\s$ on $C_{\s i}$, and by $x_\s^i=w_\s (p_i)$ the common vertex of $C_{\s i}$ and $C_\s$. Then
\begin{itemize}
\item[$(a)$] The set $\{x\in C_{\s i}:z_\s^i(x)=z_\s^i(x_\s^i)\}$ coincides with the intersection $A_\s^i$ of $C_{\s i}$ with the axis of the edge $e_{\s i}^i=w_{\s i} (e_i)$ opposite to $x_\s^i$ in $C_{\s i}$.
\item[$(b)$]  All points in $A_\s^i$ are vertices.
\end{itemize}
\end{lem}
\begin{proof}
It is not restrictive to assume that $\s=\emptyset$, $i=1$, $z(p_1):=z^1_\emptyset(p_1)=0$. We first prove the following statement.

\begin{claim} For any $n\in \bn$, denote by $\mathbf{1}_n$ the multi-index of length $n$ and taking only the value $1$, and let $\Th_n:= \set{\mathbf{1}_k: k=1,\ldots,n}$. Then,
\begin{equation}\label{decompCell}
C_1=C_{\mathbf{1}_n}\cup\bigcup_{\r\in\Th_{n-1}} C_{\r 0}\cup C_{\r 2}.
\end{equation}
\item[$(i)$] If $x\in C_{\r 0}$, $\r\in\Th_{n-1}$, and $z(x)=0$ then $x=w_{\r 0}(p_2)$, hence is on the axis $A:=A^1_\emptyset$. Analogously, if $x\in C_{\r 2}$, $\r\in\Th_{n-1}$, and $z(x)=0$ then $x=w_{\r 2}(p_0)\in A$.
\item[$(ii)$] The values of $z$ at the points $w_{\mathbf{1}_n}(p_0)$, $w_{\mathbf{1}_n}(p_2)$ are, respectively, $-1/6\cdot 5^{-n+1}$, $1/6\cdot 5^{-n+1}$.
\end{claim}
\begin{proof}[Proof of the Claim]
The statement clearly holds for $n=1$. Suppose now it is true for some $n$. Since $C_{\mathbf{1}_n}=C_{\mathbf{1}_n 0}\cup C_{\mathbf{1}_n 1}\cup C_{\mathbf{1}_n 2}$, equality \eqref{decompCell} still holds. By harmonic extension, the boundary values of $z$ on $C_{\mathbf{1}_n 0}$ are $-1/6\cdot 5^{-n+1}$, $-1/6\cdot 5^{-n}$ and $0$, hence, by the maximum principle, the value $0$ is assumed only on the vertex, proving $(i)$. The proof of $(ii)$ also follows by harmonic extension.
\end{proof}

Now we turn to the proof of the Lemma. If $z(x)=0$, either $x\in A$ or $x\in \cap_n C_{\mathbf{1}_n}$, which means $x=p_1\in A$. Conversely, if $x\in A$, either $x$ is a vertex and $z(x)=0$ or $x\in \cap_n C_{\mathbf{1}_n}$, which means $x=p_1$ hence $z(x)=0$. Both $(a)$ and $(b)$ then follow.
\end{proof}

\begin{lem}\label{lem:integer}
For any $g\in\G_n$, there exists $|\s|<n$ such that $\f_\s(g)$ is a non-vanishing integer, where $\f_\s$ is the homomorphism associated with the $\G_n$-affine potential $z_\s$.
\end{lem}
\begin{proof}
The element $g$ may be uniquely decomposed as $g=\prod_{|\t|<n}g_\t^{k_\t}$, where $g_\t$ denotes the homology class of the lacuna $\ell_\t$ according to the identification $\G_n=H_1(T_n)$. If we choose $\s$ of minimal length such that $k_\s\neq0$, we have
$$
\f_\s(g)=\sum_{|\s|\leq|\t|<n}k_\t\f_\s(g_\t)=\sum_{|\s|\leq|\t|<n}k_\t\int_{\ell_\t}dz_\s=k_\s,
$$
where we used the fact that, as observed at the beginning of Subsection \ref{sec:winding},  $\int_{\ell_\t}dz_\s$ is non-zero only if $\t\leq\s$.
\end{proof}

\begin{prop}\label{lem:sametop}
 The weak topology $\ct(z_\s)$ induced by $\{z_\s:\s\in\S\}$ on $\LL$ coincides with the projective limit topology.
\end{prop}
\begin{proof}
We shall prove that, given a point  $\tilde x\in\LL$ and one of its neighborhoods $\widetilde U$ in the projective limit topology, there exists a set $\O$, open in the  weak topology induced by $\{z_\s:\s\in\S\}$, such that $x\in\O\subseteq\widetilde  U$. This proof will in some points split in  three cases:
\begin{enumerate}
\item[(c1)] $p(\tilde x)\not\in V_*$,
\item[(c2)] $p(\tilde x)\in V_0$,
\item[(c3)] $p(\tilde x)\in V_*\setminus V_0$,
\end{enumerate}
where $p:\LL\to K$ is the covering projection. In the course of the proof, we will use the stardard notation $X^\circ$, resp. $\partial X$ for the (topological) interior, resp. boundary, of $X\subset K$. To avoid misunderstanding, we will denote by $C_\s^\iota$, resp. $b C_\s$,  the combinatorial interior, resp. boundary, of a cell $C_\s$. Observe that $C_\s^\circ = C_\s^\iota$ and $\partial C_\s = b C_\s \iff C_\s$ doesn't contain one of the vertices $p_0,p_1,p_2$.

\noindent{\bf About the neighborhood $\widetilde U$.}
By definition, there exists $n\in \bn$ such that $\widetilde U$ is the preimage in $\LL$ of a neighborhood $U$ of $x_0\in\LL_n$, where $\tilde x$ projects onto $x_0$.  It is not restrictive to assume, possibly passing to a higher covering, that
\begin{itemize}
\item[(c1-c2)]
the open set $U$  is the interior of a cell of level $n$ in $\LL_n$.
\item [(c3)]
the open set $U$ is a butterfly shaped neighborhood  made of two cells of level $n$  in $\LL_n$ in such a way that  $p_n(U)$ is not contained in a cell of level $n-1$, where $p_n:\LL_n\to K$ is the covering projection.
\end{itemize}
{\bf The choice of a fundamental domain.} As a closed fundamental domain $\cf$ in $\LL_n$, we pick a finite union of closed cells of level $n$ in $\LL_n$ such that $\cf$ is connected, $p_n(\cf)=K$ and $p_n|_{\cf^\circ}$ is injective, and with the further property that, for any $|\t|=n-1$, $p_n^{-1}(C_\t)\cap\cf$ is connected. We also require that
\begin{itemize}
\item[(c1-c2)] the neighboring cells of $U$ in $\LL_n$, whose projection to $K$ lie in the same cell of level $n-1$ containing $p_n(U)$, still belong to $\cf$. If $p_n(U) = C_{\s i}^\circ$ [$i.e.\ p_n(U) = C_{\s i}^\iota$ or $p_n(U) = C_{\s i}^\iota \cup\set{p_n(x_0)}$], we get in particular that $U$ is in the middle of the preimage $p_n^{-1}(C_\s)\cap\cf$.
\item [(c3)] same as above for the two subcells of the butterfly neighborhood $U$.  If $p_n(U)=(C_{\s i}\cup C_{\r j})^\circ$ [where, by the above assumption, $\s\neq\r$ and $i\neq j$], we get in particular that $U\cap p_n^{-1}(C_\s))$ is in the middle of the preimage $p_n^{-1}(C_\s)\cap\cf$, and $U\cap p_n^{-1}(C_\r)$ is in the middle of the preimage $p_n^{-1}(C_\r)\cap\cf$.
\end{itemize}
{\bf The normalization of the $z_\t$'s.} We have asked  the preimage in $\cf$ of any cell $C_\t$, $|\t|=n-1$ to be connected. Since such preimage consists of three cells of level $n$,  only one of them is intermediate, namely has a vertex in common with the others.  For $|\t|=n-1$, we set $z_\t$ to be zero on the third vertex of such intermediate cell, so that the range of $z_\t$ on $p_n^{-1}(C_\t)\cap\cf$ is $[-1/2,1/2]$. We normalize the $z_\t$ for $|\t|\leq n-2$ such that, again, the range of $z_\t$ on $p_n^{-1}(C_\t)\cap\cf$ is $[-1/2,1/2]$. In particular,
\begin{itemize}
\item[(c1)] the range of $z_\s$ on $U=p_n^{-1}(C_{\s i}^\iota)\cap\cf$ is $(-1/6,1/6)$, because $U$ is the intermediate cell, so that, by Lemma \ref{lem:axes}, $z_\s(x_0)\neq0$,
\item[(c2)] the range of $z_\s$ on $U=p_n^{-1}(C_{\s i}^\iota \cup\set{p_n(x_0)})\cap\cf$ is $(-1/6,1/6)$ and, by Lemma \ref{lem:axes}, $z_\s(x_0)=0$,
\item [(c3)] the ranges of $z_\s$ and $z_\r$ on $U=p_n^{-1}((C_{\s i}\cup C_{\r j})^\circ)\cap\cf$ are equal to $(-1/6,1/6)$ and, by Lemma \ref{lem:axes}, $z_\s(x_0)=z_\r(x_0)=0$.
\end{itemize}
{\bf $\cf^\circ$ is open in the topology $\ct(z_\s)$.}
By definition of $\cf$, for any $x\in \cf^\circ$, and for any $|\t|<n$, the position $z^\cf_\t(p_n(x)):=z_\t(x)$ gives a well defined function on $p_n(\cf^\circ)$.
As a consequence, with the normalization above, $z^\cf_\t$ takes values in $(-1/2,1/2)$ on the open cell $C_\t^\circ$, and is constant on the other cells, with values $-1/3,0,1/3$.
Therefore, for any $|\t|< n$, $\{z_\t(x):x\in \cf^\circ\}=(-1/2,1/2)$.
If $x\not\in\cf$, there exists $x'\in\cf$ and a non trivial $g\in\G_n$ such that $x=gx'$. By Lemma \ref{lem:integer} there exists $|\t|<n$ such that $\f_\t(g)$ is a non zero integer, hence $z_\t(x)=\f_\t(g)+z_\t(x')\in(-\infty,-1/2]\cup[1/2,+\infty)$.
Also, if $x\in\partial\cf$, $p_n(x)\in V_n\setminus V_0$, hence $\exists!\t$, $|\t|<n$ such that $p_n(x)$ is a vertex of $\ell_\t$ and $z_\t(x)=\pm1/2$. Then,
\begin{equation}
\{x\in\LL_n:z_\t(x)\in(-1/2,1/2),|\t|< n\} = \cf^\circ
\end{equation}
{\bf The construction of $\O$.}
\begin{itemize}
\item[(c1)] Set $\displaystyle \O=\bigcap_{\t\neq\s,|\t|=n-1}z_\t^{-1}(-1/2,1/2)
\cap z_\s^{-1}\{(-1/6,0)\cup(0,1/6)\}$. The result above  implies $\O\subset \cf^\circ$. By the chosen normalization, the values of $z_\s^\cf$ on the cells different from $C_\s$ can only be $-1/3$, $0$ or $1/3$, hence the values in  $(-1/6,0)\cup(0,1/6)$ are only assumed in $C_{\s i}^\circ \equiv C_{\s i}^\iota$. Therefore $\O\subset U$.
\item[(c2)] Set $\displaystyle \O=\bigcap_{|\t|=n-1}z_\t^{-1}(-1/6,1/6)$. Again $\O\subset \cf^\circ$. Since $p_n(x_0)$ is in $V_0$, the values of $z_\s^\cf$ on the cells different from $C_\s$ can only be $-1/3$ or $1/3$, namely the values $(-1/6,1/6)$ are only assumed in $C_{\s i}^\circ \equiv C_{\s i}^\iota \cup\set{p_n(x_0)}$. Therefore $\O\subset U$.
\item [(c3)] Set $\displaystyle \O=\bigcap_{|\t|=n-1}z_\t^{-1}(-1/6,1/6)$. Again $\O\subset \cf^\circ$. By construction, the removal of the cell $C_\s^\circ$ disconnects $p_n(\cf^\circ)$, and we call $D_\s(-1/3),D_\s(0),D_\s(1/3)$ the (connected) components according to the value of $z_\s^\cf$ on them. In the same way, the removal of the cell $C_\r^\circ$ disconnects  $p_n(\cf^\circ)$, and we call $D_\r(-1/3),D_\r(0),D_\r(1/3)$ the components according to the value of $z_\r^\cf$ on them.
Note that, by the simple connectedness of $T_n$, $D_\s(0)=\{p_n(x_0)\}\sqcup C_\r^\circ\sqcup D_\r(-1/3)\sqcup D_\r(1/3)$ and $D_\r(0)=\{p_n(x_0)\}\sqcup C_\s^\circ\sqcup D_\s(-1/3)\sqcup D_\s(1/3)$, where $\sqcup$ denotes disjoint union.
Then, the prescription $z_\s^\cf(y)\in(-1/6,1/6)$ selects $C^\circ_{\s i}\cup D_\s(0)$, the prescription $z_\r^\cf(y)\in(-1/6,1/6)$ selects $C^\circ_{\r j}\cup D_\r(0)$, both select $(C_{\s i}\cup C_{\r j})^\circ$, implying $\O\subset U$.
\end{itemize}
Since in all three cases $\O\in \ct(z_\s)$, we have proved the thesis.

\end{proof}


\end{document}